\documentclass[12pt]{article}
\usepackage{a4wide}
\usepackage{amsmath,amssymb,amsthm}
\usepackage[mathscr]{eucal}
\usepackage{tikz}
\usetikzlibrary{positioning,automata}
\usetikzlibrary{arrows,calc}
\usetikzlibrary{decorations.markings,plotmarks}
\usepackage{graphics}
\usepackage{color}
\usepackage{hyperref}
\usepackage{enumerate}
\usepackage{mathtools}
\usepackage{cite}

\newtheorem{theorem}{Theorem}
\newtheorem{proposition}[theorem]{Proposition}
\newtheorem{corollary}[theorem]{Corollary}
\newtheorem{lemma}[theorem]{Lemma}

\theoremstyle{definition}
\newtheorem{definition}[theorem]{Definition}

\newtheorem{problem}[theorem]{Problem}

\newtheorem{conjecture}[theorem]{Conjecture}
\newtheorem{remark}[theorem]{Remark}
\newtheorem*{notation}{Notation}

\newtheorem{obs}[theorem]{Observation}

\numberwithin{theorem}{section}

\DeclareMathOperator{\Aut}{Aut}
\DeclareMathOperator{\Iso}{Iso}
\DeclareMathOperator{\Sym}{Sym}
\DeclareMathOperator{\dist}{dist}
\DeclareMathOperator{\supp}{supp}
\DeclareMathOperator{\mindeg}{mindeg}
\DeclareMathOperator{\motion}{motion}

\newcommand{\eee}{\mathrm{e}}

\title{On the automorphism groups of
 rank-4 primitive coherent configurations}
\author{\vspace{0.3cm}Bohdan Kivva\\
University of Chicago\\
{ bkivva@uchicago.edu}}

\begin{document}
\maketitle

\vspace{-0.8cm}
\begin{abstract}
The minimal degree of a permutation group $G$ is the minimum number of points not fixed by non-identity elements of $G$. Lower bounds on the minimal degree have strong structural consequences on $G$. Babai conjectured that if a primitive coherent configuration 
   with $n$ vertices is not a Cameron scheme, then its automorphism group has
   minimal degree $\ge cn$ for some constant $c>0$. In 2014, Babai proved the desired lower bound on the minimal degree of the automorphism groups of strongly regular graphs, thus confirming the conjecture for primitive coherent configurations of rank~3. 
   
   In this paper, we extend Babai's result to primitive coherent configurations of rank 4, confirming the conjecture in this special case.   The proofs combine structural and spectral methods.

\end{abstract}

\vspace{4cm}

\section*{Acknowledgments}

The author is grateful to his advisor, Professor L\'aszl\'o Babai, for introducing him to the problems discussed in the paper and suggesting possible ways of approaching them, his invaluable assistance in framing the results, and constant support and encouragement.

\pagebreak

\tableofcontents

\pagebreak

\section{Introduction}

\subsection{Minimal degree. Motion. Our main result}

Let $\sigma$ be a permutation of a set $\Omega$. The number of points not fixed by $\sigma$ is called the \textit{degree} of the permutation $\sigma$. Let $G$ be a permutation group on the set $\Omega$. The minimum of the degrees of non-identity elements in $G$ is called the \textit{minimal degree}\footnote{For the identity permutation group on the set $\Omega$, we define its minimal degree to be $|\Omega|$, i.e., the largest possible value.} of $G$ and is denoted by $\mindeg(G)$. One of the classical problems in the theory of permutation groups is to classify the primitive permutation groups whose minimal degree is small (see \cite{Wielandt-classical}). The study of minimal degree goes back to works of Jordan \cite{Jordan}  and Bochert \cite{Bochert} in 19th century. In particular, Bochert \cite{Bochert} proved that a doubly transitive permutation group of degree $n$ has minimal degree $\geq n/4-1$ with trivial exceptions. More recently, using the Classification of Finite Simple Groups (CFSG), Liebeck and Saxl \cite{Liebeck, Liebeck-Saxl} characterized primitive permutation groups of degree $n$ with minimal degree $<n/3$ (see Theorem~\ref{Liebeck-Saxl} below).

The \textit{thickness} $\theta(G)$ of a group $G$ is the greatest $t$ for which the alternating group $A_t$ is involved as a quotient group of a subgroup of $G$ (the term was coined in \cite{Babai-str-reg}). Lower bounds on the minimal degree have strong structural consequences on $G$. In 1934 Wielandt \cite{Wielandt} showed that a linear lower bound on the minimal degree implies a logarithmic upper bound on the thickness of the group. Babai, Cameron and P\'alfy showed in \cite{thickness-primitive} that primitive permutation groups with bounded thickness have polynomially bounded order.

Switching from symmetry assumptions to regularity, in 2014 Babai \cite{Babai-str-reg, Babai-str-reg-2} proved that the minimal degree of the automorphism group of a strongly regular graph is linear in the number of vertices, with known exceptions. 

\begin{definition}\label{def-motion}
Following \cite{RS-motion}, for a combinatorial structure $\mathcal{X}$ we use term \textit{motion} to denote the minimal degree of the automorphism group $\Aut(\mathcal{X})$ 
\begin{equation}
\motion(\mathcal{X}) = \mindeg(\Aut(\mathcal{X})).
\end{equation}
\end{definition}

\begin{theorem}[Babai]\label{babai-str-reg-thm} 
Let $X$ be a strongly regular graph on $n\geq 29$ vertices. Then either
\[ \motion(X)\geq \frac{n}{8},\]
or $X$ or its complement is a triangular graph, a lattice graph, or a disjoint union of cliques of equal size.
\end{theorem}


Strongly regular graphs can be seen as a special case of a more general class of highly regular combinatorial structures called \textit{coherent configurations}, combinatorial generalizations of the orbital structure of permutation groups; the case of orbitals is called ``Schurian coherent configurations" (see Section~\ref{sec-prelim} for definitions). More specifically, strongly regular graphs are essentially coherent configurations of rank 3.

Hence, Babai's theorem confirms the following conjecture in the case of rank 3. (It was proven in \cite{Babai-annals} that a strongly regular tournament has at most $n^{O(\log(n))}$ automorphisms, and the motion is at least $n/4$.)

\begin{conjecture}[Babai]\label{conj:motion}
For every $r\geq 3$ there exists $\gamma_r>0$ such that for every  primitive coherent configuration $\mathfrak{X}$ of rank $r$ on $n$ vertices either 
$$\motion(\mathfrak{X})\geq \gamma_r n,$$
 or $\mathfrak{X}$ is a Cameron scheme.
\end{conjecture}

A stronger version of this conjecture, Conjecture~\ref{conj-3}, can be viewed as a combinatorial extension of the Liebeck--Saxl classification of primitive permutation groups~\cite{Liebeck-Saxl} (see Theorem~\ref{Liebeck-Saxl} in this paper). The proof of the Liebeck--Saxl result rests on the Classification of Finite Simple Groups through the
   O'Nan--Scott characterization of primitive groups.
   This tool is not available in our combinatorial setting.

The main result of this paper confirms Conjecture~\ref{conj:motion} for rank $4$.

\begin{theorem}[Main]\label{main-coherent}
There exists an absolute constant $\gamma_4>0$ such that for every  primitive coherent configuration $\mathfrak{X}$ of rank $4$ on $n$ vertices either 
$$\motion(\mathfrak{X})\geq \gamma_4 n,$$
 or $\mathfrak{X}$ is a Johnson scheme, or a Hamming scheme.
\end{theorem}


We proved Conjecture~\ref{conj:motion} for the class of metric schemes in \cite{kivva-drg}, \cite{kivva-geometric} (see Theorem~\ref{thm:main-motion-drg} in this paper). Metric schemes are coherent configurations with edge colors induced by a graph metric, and are essentially equivalent to distance-regular graphs.

\subsection{Coherent configurations: history and motivation}

The history of coherent configurations goes back to Schur's paper \cite{Schur} in 1933. Schur used coherent configurations to study permutation groups through their orbital configurations. Later Bose and Shimamoto \cite{Bose-Shimamoto} studied a special class of coherent configurations, called \textit{association schemes}. Coherent configurations in their full generality were independently introduced by Weisfeiler and Leman \cite{Weisfeiler}, \cite{Weisfeiler-Leman} and D. Higman \cite{Higman} around 1968. Higman developed the representation theory of coherent configurations and applied it to the permutation groups. At the same time, a related algebraic theory of coherent configurations, called ``cellular algebras,'' was introduced by Weisfeiler and Leman. Special classes of association schemes such as strongly regular graphs and, more generally, distance-regular graphs have been the subject of intensive study in algebraic combinatorics. A combinatorial study of coherent configurations was initiated by Babai in \cite{Babai-annals}. Coherent configurations play an important role in the study of the Graph Isomorphism problem, adding combinatorial divide-and-conquer tools to the arsenal. This approach was used by Babai in his recent breakthrough~\cite{Babai-GI} to prove that the Graph Isomorphism problem can be solved in quasi-polynomial time. Recently, the representation theory of coherent configurations found unexpected applications in theory of computing, specifically to the complexity of matrix multiplication \cite{Cohn-Umans}.

``Primitive coherent configurations'' are combinatorial generalizations of the orbital configurations of primitive permutation groups (see Section \ref{sec-prelim}).

Primitive permutation groups of large order were classified by Cameron \cite{Cameron} in 1981 using the Classification of the Finite Simple Groups. In particular Cameron proved that all primitive groups of order greater than $\displaystyle{n^{\log_2(n)+1}}$ are what are now called \textit{Cameron groups} (see Def. \ref{def-Cameron-groups}). \textit{Cameron schemes} are orbital configurations of Cameron groups.

A few years after Cameron's classification of large primitive groups, Liebeck \cite{Liebeck} obtained a classification of primitive groups with small minimal degree. He proved that all primitive groups of degree $n$ with minimal degree less than $n/(9\log_2 n)$ are  Cameron groups. In 1991 Liebeck and Saxl~\cite{Liebeck-Saxl} impoved the lower bound on the minimal degree of primitive non-Cameron groups to $n/3$.

\subsection{Coherent configurations: related conjectures}
In the light of the connection between primitive coherent configurations and primitive permutation groups, Babai conjectured that the following combinatorial analog of Cameron's classification should be true.

\begin{conjecture}[Babai \cite{Babai-str-reg}]\label{conj-1}
For every $\varepsilon>0$, there exists $N_{\varepsilon}$ such that if $\mathfrak{X}$ is a primitive coherent configuration on $n> N_\varepsilon$ vertices and $|\Aut(\mathfrak{X})|\geq\exp(n^{\varepsilon})$, then $\mathfrak{X}$ is a Cameron scheme.
\end{conjecture}

 The first progress on this conjecture was made in Babai's 1981 seminal paper \cite{Babai-annals}, where the conjecture was confirmed for all $\varepsilon>1/2$. As a byproduct, he solved a then 100-year-old problem on primitive, but not doubly transitive groups, giving a nearly tight bound on their order. Recently (2015), Sun and Wilmes \cite{Sun-Wilmes} made a first significant progress since that time, confirming the conjecture for all $\varepsilon>1/3$. In the case of rank 3, Chen et al. \cite{Chen-Sun-Teng} confirmed the conjecture for $\varepsilon>9/37$.

 In fact, it is believed \cite{Babai-ICM} that a much stronger version of Conjecture \ref{conj-1} is true. 
 
 \begin{conjecture}[Babai {\cite[Conjecture 12.1]{Babai-ICM}}]
There exists a polynomial $p$ such that for any  non-Cameron primitive coherent configuration $\mathfrak{X}$ on $n$ vertices the following holds.
\vspace{-0.15cm}
\begin{enumerate}[(a)]
\setlength{\itemsep}{-3pt}
 \item  $\theta(\Aut(\mathfrak{X}))\le p(\log n)$ \qquad (thickness is polylogarithmic). 
 \item  $|\Aut(\mathfrak{X})|\le \exp(p(\log n))$ \qquad (order is quasipolynomial).
\end{enumerate}
\end{conjecture}

Clearly, (a) follows from (b). As discussed above, Babai proved (a) for rank-3 coherent configurations and our Theorem \ref{main-coherent} extends this to rank 4. Using Wielandt's result (see Theorem~\ref{Wielandt-thm}),
 part~(a) follows from the following stronger version of Conjecture~\ref{conj:motion}, which is a combinatorial analog of the Liebeck-Saxl classification (see Theorem~\ref{Liebeck-Saxl}).

\begin{conjecture}[Babai]\label{conj-3}
There exists $\gamma>0$ such that if $\mathfrak{X}$ is a primitive coherent configuration on $n$ vertices and $\motion(\mathfrak{X})<\gamma n$, then $\mathfrak{X}$ is a Cameron scheme.
\end{conjecture}

\subsection{Motion lower bounds: the main tools}\label{sec-tools} 

We follow Babai's approach \cite{Babai-str-reg}. In that paper, Babai used a combination of an old combinatorial tool \cite{Babai-annals} and an introduced  spectral tool.  We used similar approach in \cite{kivva-drg, kivva-geometric}  to get bounds on the motion of distance-regular graphs of bounded diameter.

\begin{definition}
 A \textit{configuration} $\mathfrak
 {X}$ of rank $r$ on the set $V$ is a pair $(V, c)$, where $c$ is a map $c:V\times V \rightarrow \{0, 1,..., r-1\}$ such that 
\begin{enumerate}[(i)]
\item $c(v,v) \neq c(u,w)$, for every $v,u,w\in V$ with $u\neq w$,
\item for any $i<r$, there is $i^{*}<r$, such that $c(u,v) = i$ implies $c(v,u) = i^{*}$, for all $u,v \in V$.
\end{enumerate}
\end{definition}

\subsubsection{Combinatorial tool}

\begin{definition}[Babai \cite{Babai-annals}]\label{def-disting}
In a configuration $\mathfrak{X}$ a pair of vertices $u$ and $v$ is \textit{distinguished} by a vertex $x$ if the colors $c(x,u)$ and $c(x,v)$ are distinct. Define
\[ D(u,v) = \left\vert \{x\mid u,\ v \text{ are distinguished by } x\} \right\vert.\]
\end{definition}

\begin{definition}
Define the \textit{minimal distinguishing number} $D_{\min}(\mathfrak{X})$ of the configuration $\mathfrak{X} = (V, c)$ to be 
\[D_{\min}(\mathfrak{X}) = \min\limits_{u\neq v\in V} D(u, v).\]
\end{definition}

 The first tool we use for establishing motion lower bounds is the following observation.
 
\begin{obs}\label{obs1}
Let $\mathfrak{X}$ be  a configuration with $n$ vertices. Then \[\motion(\mathfrak{X})\geq D_{\min}(\mathfrak{X}).\]
\end{obs} 
\begin{proof} Indeed, let $\sigma\in \Aut(\mathfrak{X})$ be any non-trivial automorphism of $\mathfrak{X}$. Let $u$ be a vertex not fixed by $\sigma$. No fixed point of $\sigma$ distinguishes $u$ and $\sigma(u)$, so the degree of $\sigma$ is at least $ D(u, \sigma(u))\geq D_{\min}(\mathfrak{X})$.
\end{proof}

\begin{definition}
A set $S$ of vertices of a configuration $\mathfrak{X}$ is \textit{distinguishing} if every pair of distinct vertices in $\mathfrak{X}$ is distinguished by at least one element of $S$. 
\end{definition}

 Note that the pointwise stabilizer of a distinguishing set
is trivial. Thus, if $S$ is a distinguishing set of $\mathfrak{X} = (V, c)$, then $|\Aut(\mathfrak{X})|\leq n^{|S|}$, where $|V| = n$.

 In \cite{Babai-annals} Babai showed that the minimal distinguishing number can be used to bound the order of the automorphism group of a configuration.  He used the lemma below to make a dramatic improvement of the bound on the order of an uniprimitive group in terms of its degree. 
 
 \begin{lemma}[Babai {\cite[Lemma~5.4]{Babai-annals}} ]\label{Babai-disting-order-group} Let $\mathfrak{X}$ be a configuration. Then there exists a distinguishing  set of size at most $\left(2n\log n/D_{\min}(\mathfrak{X})\right)+1$. Therefore, in particular, the order of the automorphism group of $\mathfrak{X}$ satisfies $\displaystyle{|\Aut(\mathfrak{X})| \leq n^{1+2n\log(n)/D_{\min}(\mathfrak{X})}}$. 
\end{lemma}

\subsubsection{Spectral tool} 

For a $k$-regular graph $X$ let $k = \xi_1\geq \xi_2\geq...\geq \xi_n$ denote the eigenvalues of the adjacency matrix of $X$. We call $\xi = \xi(X) = \max\{ |\xi_i|: 2\leq i\leq n\}$ the \textit{zero-weight spectral radius} of~$X$. The second tool is based on the Expander Mixing Lemma.

\begin{lemma}[Babai, {\cite[Proposition~12]{Babai-str-reg}}]\label{mixing-lemma-tool}
Let $X$ be a regular graph of degree $k$ on $n$ vertices with the zero-weight spectral radius $\xi$. Suppose every pair of vertices in $X$ has at most $q$ common neighbors. Then 
\[\motion(X)\geq n\cdot \frac{(k-\xi-q)}{k}.\]
\end{lemma}

Note that this spectral tool gives trivial bound for bipartite graphs, as $\xi(X)= k$ for a $k$-regular bipartite graph $X$. A bipartite version of this lemma is established in~\cite{kivva-drg}.

\subsubsection{Structural tool}

Together with the two tools mentioned, an important ingredient of many of our proofs is Metsch's geometricity criteria (see Theorem~\ref{Metsch} and Section~\ref{sec-geom} for more details). Intuitively, Metsch's criteria states that if the number $\lambda$ of common neighbors of a pair of adjacent vertices is much larger than the number $\mu$ of common neighbors of a pair of vertices at distance 2, then the graph has a  clique geometry (see Definition \ref{def:clique-geom}).

\subsection{Outline of the proof of  the main theorem}

Primitive coherent configurations of rank 4 naturally split into three classes: configurations induced by a primitive distance-regular graph of diameter 3,  association schemes of diameter $2$ (see Definition \ref{def-assoc-diam}), and primitive coherent configurations with one undirected color and two oriented colors. The case of distance-regular graphs of diameter 3 follows from the results we proved in~\cite{kivva-geometric, kivva-drg}. Define a \textit{crown graph} to be a regular complete biprtite graph with one perfect matching deleted.

\begin{theorem}[{\cite{kivva-drg,kivva-geometric}}]\label{thm:main-motion-drg}
For any $d\geq 3$ there exists $\gamma'_d>0$, such that for any distance-regular graph $X$ of diameter $d$ with $n$ vertices either
$$\motion(X)\geq \gamma'_d n,$$
or $X$ is a Johnson graph $J(s, d)$, or a Hamming graph $H(d, s)$, or a crown graph.

\end{theorem} 

So we need to deal with the other two classes. It is not hard to see that in the case of two oriented colors, the undirected constituent is strongly regular. Thus, by Babai's result (Theorem \ref{babai-str-reg-thm}), if the number of vertices $n\geq 29$, the only possibility for $\mathfrak{X}$ to have motion less than $n/8$ is when the undirected constituent is the triagular graph $T(s)$, or the lattice graph $L_{2}(s)$, or their complement. In the latter case, we prove that the motion is linear in $n$ using the generalization of an argument appearing in the proof of {\cite[Lemma 3.5]{Sun-Wilmes}} (see Lemma~\ref{sun-wilmes-tool} in this paper).

Hence, we need to concentrate on the case of primitive association schemes of rank 4 with constituents of diameter 2. 
As the first step, we show that either we have a constituent with a $(1-\delta)$-dominant degree, or every pair of vertices can be distinguished by $\varepsilon n$ vertices (see Lemma~\ref{assoc-k2-large}). The latter directly implies that the motion is at least $\varepsilon n$. On the other hand, the fact that one of the constituents, say $X_3$, has large degree implies that some intersection numbers are quite small (see Proposition \ref{assoc-param-ineq}). This allows us to approximate the eigenvalues of the constituents $X_1$ and $X_2$, and so to approximate their zero-weight spectral radii with simple expressions involving the intersection numbers (see Lemma~\ref{assoc-spectral-radius}). We aim to apply Lemma \ref{mixing-lemma-tool} to the constituents $X_1$ and $X_2$. Considering cases how their degrees $k_1$ and $k_2$ can differ, we obtain that either the motion of $\mathfrak{X}$ is linear in $n$, or one of the graphs $X_J$ is a line graph, where $J\in\{1, 2, \{1,2\}\}$. By definition, $X_{1,2}$ is the complement of $X_3$. Since $X_1$ and $X_2$ are edge-regular, we use the classification of edge-regular and co-edge-regular graphs with smallest eigenvalue $-2$ (see Theorem \ref{geom-eig-2}). The classification tells us that either $X_i$ is strongly regular with smallest eigenvalue $-2$, or it is the line graph of a triangle-free regular graph (see Theorem~\ref{assoc-diam2} for a more precise statement). If $X_J$ is strongly regular with smallest eigenvalue $-2$, then $X_J$ is a triangular graph $T(s)$, or a lattice graph $L_2(s)$,  or has at most 28 vertices.

If one of the constituents is a line graph, this allows us to obtain more precise bounds on the intersection numbers. In particular, we  approximate the zero-weight spectral radius of the graph $X_{1,2}$ with a relatively simple expression as well. At this point, our main goal becomes to get constraints on the intersection numbers, that will allow us to apply Lemma~\ref{mixing-lemma-tool} effectively to one of the graphs $X_1$, $X_2$ or $X_{1,2}$. We consider four cases. Three of them are defined by which of the graphs $X_1$, $X_2$ or $X_{1,2}$ is strongly regular. In the fourth case, one of the constituents is the line graph of a triangle-free regular graph.  For the ranges of the parameters when Lemma \ref{mixing-lemma-tool} cannot be used effectively we use Lemma~\ref{sun-wilmes-tool}. Roughly speaking, it says that if a triangular graph $T(s)$ is a union of several constituents of a coherent configuration $\mathfrak{X}$, then $\Aut(\mathfrak{X})$ is small if the following holds for every Delsarte clique: if we look on the configuration induced on the Delsarte clique, then each pair of vertices is distinguished by a constant fraction of the vertices of the clique. The hardest case in the analysis is the case when the constituent of the smallest degree, $X_1$, is strongly regular. This case is settled in Theorem~\ref{assoc-x1-strongly-reg} and requires preparatory work with several new ideas. In particular, we use an analog of the argument from the proof of Metsch's criteria to get a constant upper bound on the fraction $k_2/k_1$ in a certain range of parameters.

\subsection{Organization of the paper} 

The paper is organized as follows. In Section \ref{sec-prelim} we introduce definitions and concepts that will be used throughout the paper.

 Section \ref{sec-coh-approx} provides bounds on the eigenvalues of the constituents of an association scheme of rank~$4$.  In Section \ref{sec-coh-reduction} we reduce the problem to association schemes with a constituent which is a line graph (and is strongly regular in most cases). Such association schemes are treated in  Section \ref{sec-subsec-str-reg}. 
 
 In Section~\ref{sec-coherent-thm-subsec} we combine all the pieces into the proof of our main result, Theorem~\ref{main-coherent}. 

Finally, Section~\ref{sec-summary} discusses some open problems related to Conjecture~\ref{conj-3}.

\section{Preliminaries}\label{sec-prelim}
In this section we introduce coherent configurations, distance-regular graphs and other related concepts that will be used throughout the paper. For more about coherent configurations we refer to \cite{Babai-GI}.

\subsection{Basic concepts and notation for graphs and groups}

\subsubsection{Graphs}

Denote $[m] =\{1,2,..., m\}$. Let $X$ be a graph. We always denote by $n$ the number of vertices of $X$, and if $X$ is regular we denote by $k$ its degree. Let $N(v)$ be the set of neighbors of a vertex $v$ in $X$ and $N_i(v) = \{ w\in X| \dist(v,w) = i\}$ be the set of vertices at distance $i$ from $v$ in the graph $X$. We denote the diameter of $X$ by $d$. If the graph is disconnected, then its diameter is defined to be $\infty$.

Denote by $\lambda = \lambda(X)$ the minimum number of common neighbors of a pair of adjacent vertices of $X$. Denote by $\mu = \mu(X)$ the maximum number of common neighbors of a pair of vertices at distance 2. Denote by $q(X)$ the maximum number of common neighbors of two distinct vertices in $X$.

Let $A$ be the adjacency matrix of the graph $X$. Suppose that $X$ is $k$-regular. Then the all-ones vector is an eigenvector of $A$ with eigenvalue $k$. We call them the \textit{trivial eigenvector} and the \textit{trivial eigenvalue}. All other eigenvalues of $A$ have absolute value not greater than $k$. We call them \textit{non-trivial} eigenvalues.

\begin{definition}
For a graph $X = (V(X), E(X))$ the \textit{line graph} $L(X)$ is defined as the graph on the vertex set $E(X)$, for which $e_1, e_2\in E(X)$ are adjacent in $L(X)$ if and only if they are incident to a common vertex in $X$.  
\end{definition}

\subsubsection{Groups}

Let $G$ be a transitive permutation group on a set $\Omega$. A $G$-invariant partition $\Omega = B_1 \sqcup B_2\sqcup\ldots\sqcup B_t$ is called a \textit{system of imprimitivity} of $G$. Any permutation group $G\leq \Sym(\Omega)$ admits two trivial $G$-invariant partitions: the particion consisting of $\Omega$ only, and the partition of $\Omega$ into singletons.  

\begin{definition}
A transitive permutation group is called \textit{primitive} if it does not admit any non-trivial system of imprimitivity.
\end{definition}

\subsection{Coherent configurations}
Our terminology follows \cite{Babai-GI}.

\noindent Let $V$ be a finite set, elements of which will be called vertices of a configuration.
\begin{definition}\label{def-conf} A \textit{configuration} $\mathfrak{X}$ of rank $r$ on the set $V$ is a pair $(V, c)$, where $c$ is a map $c:V\times V \rightarrow \{0, 1,..., r-1\}$ such that 
\begin{enumerate}[(i)]
\item\label{item-i-def-cc} $c(v,v) \neq c(u,w)$, for any $v,u,w\in V$ with $u\neq w$,
\item\label{item-ii-def-cc} for any $i<r$, there is $i^{*}<r$, such that $c(u,v) = i$ implies $c(v,u) = i^{*}$, for all $u,v \in V$.
\end{enumerate}
\end{definition}
The value $c(u,v)$ is called the \textit{color} of a pair $(u,v)$. The color $c(u,v)$ is a \textit{vertex color} if $u = v$, and is an \textit{edge color} if $u\neq v$. Then condition (i) says that edge colors are different from vertex colors, and condition (ii) says that the color of a pair $(u,v)$ determines the color of $(v,u)$. 

For every $i<r$ consider the set $R_{i} = \{(u,v): c(u,v) = i\}$ of pairs of color $i$ and consider the digraph $X_i = (V, R_i)$. We refer to both $R_i$ and $X_i$ as  the color-$i$ \textit{constituent} of $\mathfrak{X}$.
There are two possibilities: if $i = i^*$, then color $i$ and the corresponding constituent $X_i$ are called \textit{undirected}; if $i\neq i^*$, then $(i^*)^* = i$ and color $i$ together with the corresponding constituent $X_i$ are called \textit{oriented}. Clearly, $\{R_i\}_{i< r}$ forms a partition of $V\times V$. 

 We denote the adjacency matrix of the digraph $X_i$ by $A_i$. The adjacency matrices of the constituents satisfy
\begin{equation}\label{eq-constit-sum}
\sum\limits_{i = 0}^{r-1} A_i = J_{|V|} = J,
\end{equation}  
where $J$ denotes the all-ones matrix.

Note that conditions (\ref{item-i-def-cc}) and (\ref{item-ii-def-cc}) of Definition \ref{def-conf} in the matrix language mean the following. There exists a set $\mathcal{D}$ of colors, such that the identity matrix can be represented as a sum $\sum\limits_{i\in \mathcal{D}}A_i = I$. And for every color $i$, $A_i^T = A_{i^*}$.

For a set of colors $\mathcal{I}$ we denote by $X_{\mathcal{I}}$ the digraph on the set of vertices $V$, where an arc $(x,y)$ is in $X_{\mathcal{I}}$ if and only if $c(x,y) \in \mathcal{I}$. For small sets we omit braces, for example, $X_{1,2}$ will be written in place of $X_{\{1,2\}}$.
\begin{definition}
A configuration $\mathfrak{X}$ is \textit{coherent} if 
\begin{enumerate}[(i)]
\item[(iii)] for all $i,j,t<r$, there is an \textit{intersection number} $p_{i,j}^{t}$ such that, for all $u,v\in V$, if $c(u,v) = t$, then there exist exactly $p_{i,j}^{t}$ vertices $w\in V$ with $c(u,w) = i$ and $c(w,v) = j$.
\end{enumerate}
\end{definition}

The definition of a coherent configuration has several simple, but important, consequences. Let $\mathfrak{X}$ be a coherent configuration. Then every edge color is aware of the colors of its tail and head. That is, for every edge color $i$, there exist vertex colors $i_{-}$ and $i_{+}$ such that if $c(u,v) = i$, then $c(u,u) = i_{-}$ and $c(v,v) = i_{+}$. Indeed, they are the only colors for which $p_{i,i_{+}}^{i}$ and $p_{i_{-}, i}^{i}$ are non-zero. Moreover, for every color $i$ its \textit{in-degree} and \textit{out-degree} are well-defined as $k_{i}^{-} = p_{i^{*}, i}^{i_{+}}$ and $k_{i}^{+} = p_{i, i^{*}}^{i_{-}}$, respectively.

Observe that the existence of intersection numbers is equivalent to the following conditions on the adjacency matrices of the constituent digraphs.
\begin{equation}\label{eq-int-num}
A_iA_j = \sum\limits_{t = 0}^{r-1}p_{i,j}^{t}A_{t} \quad \text{ for all } i,j<r. 
 \end{equation}
 
Hence, $\{A_i: 0\leq i\leq r-1\}$ form a basis of an $r$-dimensional algebra with structure constants $p_{i,j}^t$. In particular, every $A_i$ has minimal polynomial of degree at most $r$.

\begin{definition}
A configuration $\mathfrak{X}$ is \textit{homogeneous} if $c(u,u) = c(v,v)$ for every $u,v\in V$.
\end{definition}

Unless specified otherwise, we always assume that $0$ is the vertex color of a homogeneous configuration. The constituent which corresponds to the vertex color is also referred as the \textit{diagonal constituent}. 

In a homogeneous coherent configuration we have $k_{i}^{+} = k_{i}^{-}$ for every color $i$. We denote this common value by $k_i$.

The intersection numbers of a homogeneous coherent configuration satisfy the following relations.

\begin{equation}\label{eq-sum-param}
 \sum\limits_{j = 0}^{r-1} p_{ij}^{t} = k_i \quad \text{and} \quad p_{i,j}^s k_s = p_{s,j^{*}}^{i} k_i.
 \end{equation}

Let $i,j<r$ be colors. Take $u, v\in V$ with $c(u,v) = j$. Define $\dist_i(u,v)$ to be the length $\ell$ of a shortest path $u_0 = u, u_1, ..., u_{\ell} = v$ such that $c(u_{t-1}, u_t) = i$ for $t\in [\ell]$. We claim that $\dist_{i}(j) = \dist_{i}(u,v)$ is well defined, i.e., does not depend on the choice of $u,v$, but only on the colors $j$ and $i$. 

Indeed, let $c(u,v) = c(u', v') = j$ and suppose there exist a path $u_0 = u, u_1, ..., u_{\ell} = v$ of length $\ell$, such that $c(u_{t-1}, u_t) = i$. Denote by $e_{t} = c(u_{t}, v)$. Then we know that $p_{i, e_{t}}^{e_{t-1}}\neq 0$ for $t\in [\ell-1]$. Let $u_0' = u'$. Then, as $p_{i, e_{t}}^{e_{t-1}}\neq 0$, by induction, there exists a $u_{t}'$ such that $c(u_{t-1}', u_{t}') = i$ and $c(u_{t}, v) = e_{t}$ for all $t\in [\ell-1]$. Hence, $\dist(u',v')\leq \dist(u,v)$ and similarly $\dist(u,v)\leq \dist(u', v')$. Therefore, $\dist_i(j)$ is well-defined.

\begin{obs}\label{obs-color-dist}
If $\dist_{i}(j)$ is finite, then $\dist_{i}(j)\leq r-1$.
\end{obs}
\begin{proof}
 Suppose that $\dist_{i}(j)$ is finite, then for $c(u,v) = j$ there exists a shortest path $u_0 = u, u_1, ..., u_{\ell} = v$ with $c(u_{t-1}, u_t) = i$. Denote by $e_{t} = c(u_{t}, v)$ for $0\leq t\leq \ell-1$. Then, all $e_{t}$ are distinct edge colors, or the path can be shortened. Thus $\ell \leq r-1$.
\end{proof}

\begin{definition}
A coherent configuration is called an \textit{association scheme} if $c(u,v) = c(v,u)$ for any $u,v\in V$. 
\end{definition}

\begin{corollary}
Any association scheme is a homogeneous configuration.
\end{corollary}
\begin{proof}
Since in a coherent configuration color of every edge is aware of the colors of its head and tail vertices, these vertices have the same color for every edge. 
\end{proof}

Note, for an association scheme every constituent digraph is a graph. Thus, for an association scheme and $i\neq 0$, the $i$-th constituent $X_i$  is a $k_i$-regular graph with $\lambda(X_i) = p_{i,i}^{i}$.
Moreover, it is clear that $p_{i,j}^s = p_{j,i}^s$ for association schemes.

\begin{definition}
A homogeneous coherent configuration is called \textit{primitive} if every constituent is strongly connected.
\end{definition}

It is not hard to check that every constituent graph of a homogeneous coherent configuration is stongly connected if and only if it is weakly connected.

Note, that by Observation \ref{obs-color-dist}, we have $\dist_{i}(j)\leq r-1$ for any edge colors $i,j$ of a primitive coherent configuration.

The following definition will be useful.

\begin{definition}\label{def-assoc-diam} We say that an association scheme has \textit{diameter} $d$ if every non-diagonal constituent has diameter at most $d$ and there exists a non-diagonal constituent of diameter~$d$.
\end{definition}

Note, that if an association scheme has finite diameter, then in particular it is  primitive. Alternatively, every primitive association scheme of rank $r$ has diameter $\leq r-1$. 

\begin{definition}
A regular graph is called \textit{edge-regular} if every pair of adjacent vertices has the same number of common neighbors. A graph is called \textit{co-edge-regular} if its complement is edge-regular. 
\end{definition}

Observe, that for every undirected color $i$ the constituent $X_i$ is an edge-regular graph.
We also introduce the following definition.

\begin{definition}
We say that a homogeneous coherent configuration $\mathfrak{X}$ of rank $r$ \textit{has constituents ordered by degree}, if color $0$ corresponds to the diagonal constituent and the degrees of non-diagonal constituents satisfy $k_1\leq k_2\leq ...\leq k_{r-1}$. 
\end{definition}

\subsection{Triangle inequality for distinguishing numbers}\label{sec-dist-number-gen-drg}

It is easy to see that for a homogeneous coherent configuration $\mathfrak{X}$, the number $D(u,v)$ of vertices which distinguish $u$ and $v$  (see Def.~\ref{def-disting}) depends only on the color $i$ between $u$ and $v$. So one can define $D(i) = D(u,v)$. We need the following lemma by Babai \cite{Babai-annals}.

\begin{lemma}[Babai {\cite[Proposition 6.4]{Babai-annals}}]\label{babai-dist}
Let $\mathfrak{X}$ be a homogeneous coherent configuration of rank $r$. Then for any colors $1\leq i, j\leq r-1$ the inequality $D(j)\leq \dist_{i}(j) D(i)$ holds.
\end{lemma}
\begin{proof}
Since the constituent $X_i$ is connected, statement follows from the triangle inequality  $D(u,v)\leq D(v,w)+D(w, u)$ for any  vertices $u,v,w$.
\end{proof}

\subsection{Distance-regular graphs}
\begin{definition}
 A connected graph $X$ of diameter $d$ is called \textit{distance-regular} if for every $0\leq i\leq d$ there exist constants $a_i, b_i, c_i$ such that for all $v\in X$ and all $w\in N_i(v)$ the number of edges between $w$ and $N_i(v)$ is $a_i$, between $w$ and $N_{i+1}(v)$ is $b_i$, and between $w$ and $N_{i-1}(v)$  is $c_i$. The sequence
 \[\iota(X) = \{b_0, b_1,\ldots, b_{d-1}; c_1, c_2,\ldots, c_d\}\]
 is called the \textit{intersection array} of $X$.
\end{definition}

Note, that for a distance-regular graph $b_d = c_0 = 0$, $b_0 = k$, $c_1 = 1$, $\lambda = a_1$ and $\mu = c_2$. By edge counting, the following straightforward properties of the parameters of a distance-regular graph hold.
\begin{enumerate}
\item $a_i+b_i+c_i = k$, \quad  for every $0\leq i\leq d$,
\item $|N_i(v)|b_i = |N_{i+1}(v)|c_{i+1}$, \quad $\Rightarrow \quad $ $k_i := |N_i(v)|$ does not depend on vertex $v\in X$.
\item $b_{i+1} \leq b_{i}$ and $c_{i+1}\geq c_{i}$, \quad for $0\leq i\leq d-1$.
\end{enumerate}

With any graph of diameter $d$ we can naturally associate matrices $A_{i}$, where rows and columns are indexed by vertices, with $(A_i)_{u,v} = 1$ if and only if $\dist(u,v) = i$, and $(A_i)_{u,v} = 0$ otherwise. That is, $A_i$ is the adjacency matrix of the distance-$i$ graph $X_i$ of $X$. For a  distance-regular graph these matrices satisfy the following relations
\begin{equation}
A_0 = I,\quad A_1 =: A,\quad \sum\limits_{i=0}^d A_i = J,
\end{equation}
\begin{equation}\label{eq-rec}
 AA_i = c_{i+1}A_{i+1}+a_iA_i+b_{i-1}A_{i-1}\quad \text{for } 0\leq i\leq d,
 \end{equation}  
where $c_{d+1} = b_{-1} = 0$ and $A_{-1} = A_{d+1} = 0$. Clearly, Eq.~\eqref{eq-rec} implies that for every $0\leq i\leq d$ there exists a polynomial $\nu_i$ of degree exactly $i$, such that $A_i = \nu_i(A)$. Moreover, minimal polynomial of $A$ has degree exactly $d+1$. Hence, since $A$ is symmetric, $A$ has exactly $d+1$ distinct real eigenvalues. Additionally, we conclude that for all $0\leq i,j,s\leq d$ there exist \textit{intersection numbers} $p_{i,j}^{s}$, such that 
\[A_iA_j = \sum\limits_{s = 0}^{d}p_{i,j}^{s} A_s.\]
Recalling the definition of $A_i$, this implies that for any $u,v\in X$ with $\dist(u,v) = s$ there exist exactly $p_{i,j}^{s}$ vertices at distance $i$ from $u$ and distance $j$ from $v$, i.e., $|N_i(u)\cap N_j(v)| = p_{i,j}^{s}$. 

Therefore, every distance-regular graph $X$ of diameter $d$ induces an association scheme $\mathfrak{X}$ of rank $d+1$, where vertices are connected by an edge of color $i$ in $\mathfrak{X}$ if and only if they are at distance $i$ in $X$, for $0\leq i\leq d$.  Hence, we get the following statement.
\begin{lemma}\label{lem:drg-metric}
If a graph $X$ is distance-regular of diameter $d$, then the distance-$i$ graphs $X_i$ form constituents of an association scheme $\mathfrak{X}$ of rank $d+1$ and diameter $d$. In the opposite direction, if an association scheme of rank $d+1$ has a constituent of diameter $d$, then this constituent is distance-regular. 
\end{lemma}

\subsection{Weisfeiler-Leman refinement}

Let $\mathfrak{X} = (V, c)$ be a configuration. A \textit{refinement} of the coloring $c$ is a new coloring $c'$ also defined on $V\times V$, such that if $c'(x) = c'(y)$ for $x,y\in V\times V$, then $c(x) = c(y)$. If coloring $c'$ satisfies condition (ii) of Definition \ref{def-conf}, then $\mathfrak{X}' = (V, c')$ is a refined configuration.

An important example of a refinement was introduced by Weisfeiler and Leman in \cite{Weisfeiler-Leman} in 1968. The \textit{Weisfeiler-Leman refinement} proceeds in rounds. On each round it takes a configuration $\mathfrak{X}$ of rank $r$ and for each pair $(x,y)\in V\times V$ it encodes in a new color $c'(x, y)$ the following information: the color $c(x,y)$,  and the numbers $wl_{i,j}(x, y) = |\{z: c(x,z) = i, c(z,y) = j\}|$ for all $i,j\leq r$. That is, pairs $(x_1,y_1)$ and $(x_2, y_2)$ receive the same color if and only if their colors and all the numbers $wl_{i,j}$ were equal. It is easy to check that for the refined coloring $c'$, the structure $\mathfrak{X}' = (V, c')$ is a configuration as well. The refinement process applied to a configuration $\mathfrak{X}$ takes $\mathfrak{X}$ as an input on the first round, and on every subsequent round in takes as an input the output of the previous round . The refinement process stops when it reaches a stable configuration (i.e, $\mathfrak{Y}' = \mathfrak{Y}$). It is easy to see that the process will always stop as the number of colors increases after a refinement round applied to a non-stable configuration. One can check that configurations that are stable under this refinement process are precisely coherent configurations. Therefore, the Weisfeiler-Leman refinement process takes any configuration and refines it to a coherent configuration.

The Weisfeiler-Leman refinement is \textit{canonical} in the following sense. Let $\mathfrak{X}$ and $\mathfrak{Y}$ be configurations and $\mathfrak{X}^{wl}, \mathfrak{Y}^{wl}$ be the refined configurations outputted by the process. Then the sets of isomorphisms are equal
\[\Iso(\mathfrak{X}, \mathfrak{Y}) = \Iso(\mathfrak{X}^{wl}, \mathfrak{Y}^{wl}) \subseteq \Sym(V).\]

Another important procedure that can be applied to a configuration is individualization. We say that a configuration $\mathfrak{X}^* = (V, c^*)$ is an \textit{individualization} of $\mathfrak{X}$ at pair $(x,y)\in V\times V$, if $c^*$ is obtained from $c$ by replacing $c(x,y)$ and $c(y,x)$ with new (possibly equal) colors $c^*(x,y)$ and $c^{*}(y,x)$, which were not in the image of $c$. A configuration individualized at a pair $(x,x)$ for $x\in V$ is said to be individualized at the vertex $x$. Similarly, we say that $\mathfrak{X}$ is individualized at a set of vertices $S\subseteq V$ if it is subsequently individualized in each vertex from $S$.

\begin{definition}
Let $\mathfrak{X}$ be a configuration and $S\subseteq V$ be a subset of vertices. Denote by $\mathfrak{X}'$ the configuration obtained from $\mathfrak{X}$ by individualization at the set $S$. We say that $S$ \textit{splits $\mathfrak{X}$ completely} with respect to some canonical refinement process $\mathfrak{r}$, if this refinement process applied to $\mathfrak{X}'$ stabilizes only when every vertex from $V$ gets a unique color.
\end{definition}

We will use this definition only with respect to the Weisfeiler-Leman refinement process, so we will omit mentioning that in the future.

Observe, that if $S$ splits $\mathfrak{X}$ completely, then the pointwise stabilizer $\Aut(\mathfrak{X})_{(S)}$ is the identity group. Thus, in particular $|\Aut(\mathfrak{X})|\leq n^{|S|}$, where $n = |V|$. Hence, individualization/refinement techniques can be used to bound the order of the automorphism group.

\subsection{Johnson, Hamming and Cameron schemes}\label{sec-subsec-exceptions}

In this subsection we define families of graphs and coherent configurations with huge automorphism groups. We show that for a certain range of parameters they have motion sublinear in the number of vertices.
\subsubsection{Johnson schemes}
\begin{definition}
Let $d\geq 2$ and $\Omega$ be a set of $m\geq 2d$ points. The \textit{Johnson graph} $J(m,d)$ is the graph on the set $V(J(m,d)) = \binom{\Omega}{d}$ of $n = \binom{m}{d}$ vertices, for which two vertices are adjacent if and only if the corresponding subsets $U_1, U_2\subseteq \Omega$ differ by exactly one element, i.e., $|U_1\setminus U_2| = |U_2\setminus U_1| = 1$.
\end{definition}

It is not hard to check that $J(m, d)$ is a distance-regular graph of diameter $d$ with the intersection numbers
\[b_i = (d-i)(m-d-i) \quad \text{and} \quad c_{i+1} = (i+1)^{2}, \quad \text{for } 0\leq i<d.\] 
In particular, $J(m, d)$ is regular of degree $k = d(m-d)$ with $\lambda = m-2$ and $\mu = 4$.  The eigenvalues of $J(m,d)$ are
\[ \xi_j = (d-j)(m-d-j)-j \quad \text{with multiplicity}\quad \binom{m}{j} - \binom{m}{j-1}, \, \text{for } 0\leq j\leq d.\]

The automorphism group of $J(m,d)$ for $m\geq 2d+1$ is the induced symmetric group $S_m^{(d)}$, which acts on $\binom{\Omega}{d}$ via the induced action of $S_m$ on $\Omega$. Indeed, it is clear, that $S^{(d)}_m\leq \Aut(J(m,d))$. The opposite inclusion can be derived from the Erd\H{o}s-Ko-Rado theorem.

Thus, for a fixed $d$ we get that the order is $|\Aut(J(m,d)| = m! = \Omega(\exp(n^{1/d}))$, the thickness satisfies $\theta(\Aut(J(m, d))) = m = \Omega(n^{1/d})$ and  
$$\motion(J(m,d)) = O(n^{1-1/d}).$$

\begin{definition} The association scheme $\mathfrak{J}(m,d)$ induced by $J(m,d)$ is called \textit{Johnson scheme}.
\end{definition}

One can check that $\mathfrak{J}(m,d)$ is primitive for $m\geq 2d+1$ and $\Aut(\mathfrak{J}(m,d)) =\Aut(J(m,d))$. 

\begin{definition}
 The Johnson graph $J(s,2)$ is called \textit{triangular graph} and is denoted by $T(s)$, where $s\geq 4$.
\end{definition} 

\subsubsection{Hamming schemes}

\begin{definition}
Let $\Omega$ be a set of $m\geq 2$ points. The \textit{Hamming graph} $H(d, m)$ is the graph on the set $V(H(d, m)) = \Omega^{d}$ of $n = m^d$ vertices, for which a pair of vertices is adjacent if and only if the corresponding $d$-tuples $v_1, v_2$ differ in precisely one position (in other words, if the Hamming distance $d_H(v_1, v_2)$ for the corresponding tuples equals 1).
\end{definition}
Again, it is not hard to check that $H(d,m)$ is a distance-regular graph of diameter $d$ with the intersection numbers
\[ b_i =(d-i)(m-1) \quad \text{and} \quad c_{i+1} = i+1, \quad \text{for } 0\leq i\leq d-1.\]
In particular, $H(d,m)$ is regular of degree $k = d(m-1)$ with $\lambda = m-2$ and $\mu = 2$. The eigenvalues of $H(d, m)$ are
\[ \xi_j = d(m-1) - jm \quad \text{with multiplicity}\quad \binom{d}{j}(m-1)^{j}, \, \text{for } 0\leq j\leq d.\]

The automorphism group of $H(d, m)$ is isomorphic to the wreath product $S_m\wr S_d$. Hence,  its order is $|\Aut(H(d, m))| = (m!)^d d!$, the thickness satisfies $\theta(H(d,m)) \geq m = n^{1/d}$ and $$\motion(H(d, m))\leq 2m^{d-1} = O(n^{1-1/d}).$$

\begin{definition}
The association scheme $\mathfrak{H}(d,m)$ induced by the Hamming graph $H(d,m)$ is called \textit{Hamming scheme}.
\end{definition}


Again, one can check that $\mathfrak{H}(d,m)$ is a primitive coherent configuration for $m\geq 3$ and $\Aut(\mathfrak{H}(d, m)) = \Aut(H(d,m))$.

As for Johnson graphs, the graph $H(2, m)$ has a special name and notation.
\begin{definition}
The Hamming graph $H(2, m)$ is called the \textit{lattice graph} and is denoted by $L_2(m)$, where $m\geq 2$.
\end{definition}

\subsubsection{Cameron groups and schemes}

Let $G\leq \Sym(\Omega)$ be a permutation group. The orbits of the $G$-action on $\Omega\times\Omega$ are called \textit{orbitals} of $G$. 

\begin{definition}\label{def-orbital}
For $G\leq Sym(\Omega)$ with orbitals $R_1, R_2, ..., R_k$ define \textit{orbital configuration} as $\mathfrak{X}(G) = (\Omega, \{R_1, R_2, ..., R_k\})$. A configuration is called \textit{Schurian} if it is equal to $\mathfrak{X}(G)$ for some permutation group $G$.
\end{definition}

Observe that the configuration $\mathfrak{X}(G)$ is coherent. By the discussion above, Johnson and Hamming schemes are Schurian configurations.

Using the Classification of the Finite Simple Groups, Cameron in \cite{Cameron} classified all primitive groups of degree $n$ of order at least $n^{c\log\log n}$. We state here Mar\'oti's refined version of this result.

\begin{definition}[Cameron groups]\label{def-Cameron-groups}
Let $G$ be a primitive group. Let $m,k,d$ be positive integers and $(A_{m}^{(k)})^{d}\leq G\leq S_{m}^{(k)}\wr S_d$, where $S_{m}^{(k)}\wr S_d$ has the product action on $\binom{m}{k}^d$ elements. Then the group $G$ is called a \textit{Cameron group}.

 The corresponding orbital configuration $\mathfrak{X}(G)$ is called a \textit{Cameron scheme}.

\end{definition}

\begin{theorem}[Cameron \cite{Cameron}, Mar\'oti \cite{Maroti}]\label{cameron-maroti}
If $G$ is a primitive permutation group of degree $n>24$, then one of the following is true.
\begin{enumerate}
\item $G$ is a Cameron group.
\item $|G|\leq n^{1+\log(n)}$. 
\end{enumerate}
\end{theorem}

The Cameron groups appear as an exceptions in other similar classifications of ``large'' primitive groups. We mention the classification result by Liebeck and Saxl \cite{Liebeck-Saxl}.


\begin{theorem}[Liebeck, Saxl \cite{Liebeck-Saxl}]\label{Liebeck-Saxl}
If $G$ is a primitive permutation group of degree $n$, then one of the following is true.
\begin{enumerate}
\item $G$ is a Cameron group.
\item $\mindeg(G)\geq n/3$. 	 
\end{enumerate}
\end{theorem} 

The next lemma shows that in a certain range of parameters Cameron groups have sublinear minimal degree.

\begin{lemma}\label{cameron-min-deg}
Let $G$ be a Cameron group with $(A_m^{k})^d\leq G\leq S_m^{(k)}\wr S_d$ which acts on $n = \binom{m}{k}^d$ points, where $k\leq m/2$. Then as $m\rightarrow \infty$, the following holds uniformly in $d$: we have $\mindeg(G) = o(n)$ if and only if $k = o(m)$. 
\end{lemma}
\begin{proof}
It is not hard to see that the minimal degree of $G$ is realized by the induced action of a cycle of length 2 or 3 (in $S_m$ or $A_m$) on $k$-subsets in just one of $d$ coordinates.
If there is a 2-cycle action in a coordinate, then the minimal degree of $G$ is
\[\left(\binom{m}{k} - \binom{m-2}{k} - \binom{m-2}{k-2}\right)\binom{m}{k}^{d-1},\]
otherwise, the minimal degree of $G$ is
\[\left(\binom{m}{k} - \binom{m-3}{k} - \binom{m-3}{k-3}\right)\binom{m}{k}^{d-1}.\]
As $m\rightarrow \infty$ these expressions are equal
\[ n\cdot\left(1-\frac{(m-k)^2+k^2}{m^2}+o(1)\right) \quad \text{and} \quad  n\cdot \left(1-\frac{(m-k)^3+k^3}{m^3}+o(1)\right),\]
respectively. Clearly, each of these expressions is $o(n)$ if and only if $k = o(m)$. 

\end{proof}

\subsection{Wielandt's upper bound for thickness}

In 1934, Wielandt~\cite{Wielandt} proved that linear lower bounds for minimal degree of permutation groups imply logarithmic upper bound for the thickness of a group.

\begin{theorem}[Wielandt \cite{Wielandt}, see {\cite[Theorem 6.1]{Babai-doublytransitive}}]\label{Wielandt-thm}
Let $n>k>\ell$ be positive integers, $k\geq 7$, and let $0<\alpha<1$. Suppose that $G$ is a permutation group of degree $n$ and minimal degree at least $\alpha n$. If
\[\ell(\ell-1)(\ell-2)\geq (1-\alpha)k(k-1)(k-2),\]
and $\theta(G)\geq k$, then $n\geq \binom{k}{\ell}$.
\end{theorem} 

\begin{corollary}\label{Wielandt-cor}
Let $G$ be a permutation group of degree $n$. Suppose $\mindeg(G)\geq \alpha n$. Then the thickness $\theta(G)$ of $G$ satisfies $\displaystyle{\theta(G)\leq \frac{3}{\alpha}\ln(n)}$.
\end{corollary}

\subsection{Metsch's criteria for clique geometry}\label{sec-geom}

In this section we discuss graphs that contain a special rich combinatorial structure, called clique geometry.

\begin{definition}\label{def:clique-geom}
We say that a graph $X$ contains a \textit{clique geometry}, if there exists a collection $\mathcal{C}$ of maximal (by inclusion) cliques, such that every edge is contained in exactly one clique of $\mathcal{C}$.
The cliques of $\mathcal{C}$ sometimes are called \textit{lines}.
\end{definition}

 Metsch proved that any connected graph satisfying quite simple constraints contains a clique geometry.

\begin{theorem}[Metsch {\cite[Result 2.2]{Metsch}}]\label{Metsch}
Let $\mu\geq 1$, $\lambda^{(1)}, \lambda^{(2)}$ and $m$ be integers. Assume that $X$ is a connected graph with the following properties.
\begin{enumerate}
\item Every pair of adjacent vertices has at least $\lambda^{(1)}$ and at most $\lambda^{(2)}$ common neighbors;
\item Every pair of non-adjacent vertices has at most $\mu$ common neighbors;
\item $2\lambda^{(1)}-\lambda^{(2)}>(2m-1)(\mu-1)-1$;
\item Every vertex has degree less than $(m+1)(\lambda^{(1)}+1)-\frac{1}{2}m(m+1)(\mu-1)$.
\end{enumerate}

Define a \emph{line} to be a maximal clique $C$ satisfying $|C|\geq \lambda^{(1)}+2-(m-1)(\mu-1)$. Then every vertex is in at most $m$ lines, and every pair of adjacent vertices lies in a unique line. 

\end{theorem}

\begin{lemma}[see {\cite[Prop. 9.8]{Koolen-survey}}]\label{geom-smallest-eigenvalue} Suppose that $X$ satisfies the conditions of the previous theorem. Then the smallest eigenvalue of $X$ is at least $-m$.
\end{lemma}
\begin{proof}
Let $\mathcal{C}$ be the collection of lines of $X$. Consider $|V|\times |\mathcal{C}|$ vertex-clique incidence matrix $N$. That is, $(N)_{v,C} = 1$ if and only if $v\in C$ for $v\in X$ and $C\in \mathcal{C}$. Since every edge belongs to exactly one line, we get $NN^{T} = A+D$, where $A$ is the adjacency matrix of $X$ and $D$ is a diagonal matrix. Moreover, $(D)_{v,v}$ equals to the number of lines that contain $v$. By the previous theorem, $D_{v,v} \leq m$ for every $v\in X$. Thus, 
\[A+mI = NN^{T}+(mI-D)\] 
is positive semidefinite, so all eigenvalues of $A$ are at least $-m$.
\end{proof}

The following lemma follows from e.g. \cite{regular-classif}, we include an elementary proof for completeness.
\begin{lemma}\label{m2-line}
Suppose $X$ is a non-complete regular graph that satisfies the conditions of Theorem~\ref{Metsch} with $m=2$. Then $X$ is a line graph $L(Y)$ for some graph $Y$.
\end{lemma}
\begin{proof}
Let $\mathcal{C}$ be a collection of lines of $X$. Observe that since $X$ is regular and non-complete, every vertex is in at least two elements of $\mathcal{C}$. Indeed, if $v$ belongs to the unique $C\in \mathcal{C}$, then $N(v) = C$, as every edge of $X$ lies in the unique element of $\mathcal{C}$. Thus, regularity of $X$ implies that $N(v)$ is a connected component of $X$, which by assumptions is connected and non-complete. Therefore, by Theorem \ref{Metsch} every vertex of $X$ is in exactly two lines.

Define a graph $Y$ with the set of vertices $V(Y) = \mathcal{C}$. Two vertices $C_1, C_2 \in \mathcal{C}$ in $Y$ are adjacent if and only if $C_1\cap C_2 \neq \emptyset$. We claim that $L(Y) \cong X$. Indeed, since every edge of $X$ is in exactly one line, $|C_1\cap C_2| \leq 1$, so there is a well-defined map $f:E(Y)\rightarrow V(X)$. Moreover, since every vertex of $X$ is in exactly two lines, the map $f$ is bijective. Additionally, a pair of edges in $Y$ shares the same vertex if and only if there is an edge between the corresponding vertices in $X$. Hence, $L(Y) \cong X$. 
\end{proof}

In the case of distance-regular graphs, a special class of graphs with a clique geometry is distinguished. Let $X$ be a distance-regular graph, and $\theta_{\min}$ be its smallest eigenvalue. Delsarte proved in \cite{Delsarte} that every clique $C$ in $X$ satisfies $|C|\leq 1-k/\theta_{\min}$. A clique in $X$ of size $1-k/\theta_{\min}$ is called a \textit{Delsarte clique}.

\begin{definition}
A distance-regular graph $X$ is called \textit{geometric} if there exists a clique geometry $\mathcal{C}$ such that every clique $C\in \mathcal{C}$ is a Delsarte clique.
\end{definition}

More on geometric distance-regular graphs can be found in~\cite{Koolen-survey}.

Existance of a clique geometry provides the following useful bound on the $\mu(X)$.
\begin{lemma}\label{clique-mu-bound}
Let $X$ be a graph. Let $\mathcal{C}$ be a collection of cliques in $X$, such that every edge lies in a unique clique from $\mathcal{C}$ and every vertex is in at most $m$ cliques from $\mathcal{C}$. Then $\mu(X)\leq m^2$. 
\end{lemma}
\begin{proof}
Let $u,v\in X$ be a pair of vertices at distance $2$. By the assumptions of the lemma we can write $N(u) = \bigcup\limits_{i = 1}^{m_u} C_i^u$ and $N(v) = \bigcup\limits_{i = 1}^{m_v} C_i^v$, where $C_i^u, C_j^v\in \mathcal{C}$. Since $\dist(u,v) = 2$, all cliques are distinct. Observe, that any pair of distinct cliques in $\mathcal{C}$ intersect each other in at most one vertex. Hence, $N(u)\cap N(v)\leq m_u m_v$, so $\mu(X)\leq m^2$. 
\end{proof}

Our approach is to show that the graph $X$, for which Lemma~\ref{mixing-lemma-tool} and Observation~\ref{obs1} are not applicable, satisfies Metsch's criteria for small $m$. This gives a hard structural constraints on $X$. For instance, if $m=2$, then $X$ is necessarily a line graph, as shown in Lemma~\ref{m2-line}. Moreover, by Lemma~\ref{geom-smallest-eigenvalue}, the smallest eigenvalue of $X$ is at least $-m$. 

Therefore, the problem of classifying graphs with certain level of regularity and bounded smallest eigenvalue is of great importance for our technique. One of the first important results of this flavor is due to Seidel, who characterized the strongly regular graphs with smallest eigenvalue $-2$. We state here a generalization of Seidel's result given by Brouwer, Cohen and Neumaier, which will be used later in Section \ref{sec-coh-reduction}.

\begin{definition}
 A graph $X$ is called an \textit{$m\times n$-grid} if it is the line graph of the complete bipartite graph $K_{m,n}$.
\end{definition}

\begin{theorem}[{\cite[Corollary 3.12.3]{BCN}}]\label{geom-eig-more2}
Let $X$ be a connected regular graph with smallest eigenvalue $>-2$. Then $X$ is a complete graph, or $X$ is an odd polygon.  
\end{theorem}

\begin{theorem}[Seidel \cite{Seidel}; Brouwer-Cohen-Neumaier {\cite[Theorem 3.12.4]{BCN}}]\label{geom-eig-2}
Let $X$ be a connected regular graph on $n$ vertices with smallest eigenvalue $-2$.
\begin{enumerate}[(i)]
\item If $X$ is strongly regular, then $X = L_{2}(s)$, or $X = T(s)$ for some $s$, or $n\leq 28$.
\item If $X$ is edge-regular, then $X$ is strongly regular or the line graph of a regular triangle-free graph.
\item If $X$ is co-edge-regular, then $X$ is strongly regular, an $m_1\times m_2$-grid, or one of the two regular subgraphs of the Clebsh graph with $8$ and $12$ vertices, respectively.
\end{enumerate}
\end{theorem}

\subsection{Approximation tool}
In Section \ref{sec-coh-approx} we will use the following results from the approximation theory which allow us to estimate the roots of a perturbed polynomial.

\begin{theorem}[{\cite[Appendix A]{Ostrowski}}]\label{poly-approx}
Let $n\geq 1$ be an integer. Consider two polynomials of degree $n$
\[ f(x) = a_0x^n+...+a_{n-1}x+a_n, \quad g(x) = b_0x^n+...+b_{n-1}x+b_n,\]
where $a_0 = b_0 = 1$. Denote $M = \max  \{|a_i|^{1/i}, |b_i|^{1/i}:  0\leq i\leq n\}$ and
\[\varepsilon = 2n \left( \sum\limits_{i = 1}^{n} |b_i - a_i| (2M)^{n-i}\right)^{1/n}.\]
Let $x_1, x_2, ..., x_n$ denote the roots of $f$ and $y_1, y_2, ..., y_n$ denote the roots of $g$. Then, there exists a permutation $\sigma \in S_n$ such that for every $1\leq i\leq n$
\[|x_i - y_{\sigma(i)}|\leq \varepsilon.\]

\end{theorem}

\section{Approximation of the eigenvalues of the constituents}\label{sec-coh-approx}

In this subsection we provide technical lemmas that allow us to approximate the zero-weight spectral radius of $X_1$, $X_2$ and $X_{1,2}$ under quite modest assumptions.

\begin{lemma}\label{rank-4-roots}
Let $\mathfrak{X}$ be an association scheme of rank 4. Let $\eta$ be a non-trivial eigenvalue of $A_1$. Then $\eta$ satisfies cubic polynomial equation $\eta^3+a_1\eta^2+a_2\eta+a_3 = 0$, where
\[a_1 =  - (p_{1,1}^{1}+p_{1,2}^{2}-p_{1,1}^{3}-p_{1,2}^{3})\quad \quad a_3 = \left( (p_{1,2}^{2}-p_{1,2}^{3})(k_1-p_{1,1}^{3})+(p^{2}_{1,1} - p_{1,1}^{3})p_{1,2}^{3}\right)\]
\[a_2 = \left( (p_{1,2}^{2}-p_{1,2}^{3})(p^{1}_{1,1} - p_{1,1}^{3}) - (p^{2}_{1,1} - p_{1,1}^{3})(p_{1,2}^{1}-p_{1,2}^{3})- (k_1-p^{3}_{1,1})\right)\]
\end{lemma}
\begin{proof}
By Eq. \eqref{eq-int-num} for intersection numbers we have
\[A_1^{2} = p^{1}_{1,1}A_1+p_{1,1}^{2}A_2+p_{1,1}^{3}A_3+k_1I.\]
We can eliminate $A_3$ using Eq. \eqref{eq-constit-sum}.
\begin{equation}\label{eq-eig-comp-1}
A_1^{2}= (p^{1}_{1,1} - p_{1,1}^{3})A_1+(p_{1,1}^{2}-p_{1,1}^{3})A_2+(k_1-p_{1,1}^{3})I+p_{1,1}^{3}J.
\end{equation}
Let us multiply previous equation by $A_1$ and use Eq. \eqref{eq-int-num}.
\begin{equation}\label{eq-eig-comp-2}
\begin{gathered}
A_1^{3} = (p^{1}_{1,1} - p_{1,1}^{3})A_1^{2}+(k_1-p^{3}_{1,1})A_1+p_{1,1}^{3}k_1J+\\
+(p^{2}_{1,1} - p_{1,1}^{3})((p_{1,2}^{1}-p_{1,2}^{3})A_1+(p_{1,2}^{2}-p_{1,2}^{3})A_2+p_{1,2}^{3}J-p_{1,2}^{3}I).
\end{gathered}
\end{equation}
Combining Eq. \eqref{eq-eig-comp-1} and \eqref{eq-eig-comp-2} we eliminate $A_2$ as well.

\[A_1^{3} - (p_{1,2}^{2}-p_{1,2}^{3})A_1^2 = (p^{1}_{1,1} - p_{1,1}^{3})A_1^{2}+(k_1-p^{3}_{1,1})A_1+p_{1,1}^{3}k_1J -  \]
\[ - (p_{1,2}^{2}-p_{1,2}^{3})(p^{1}_{1,1} - p_{1,1}^{3})A_1 - (p_{1,2}^{2}-p_{1,2}^{3})((k_1-p_{1,1}^{3})I+p_{1,1}^{3}J)+\]
\[+(p^{2}_{1,1} - p_{1,1}^{3})(p_{1,2}^{1}-p_{1,2}^{3})A_1+(p^{2}_{1,1} - p_{1,1}^{3})(p_{1,2}^{3}J-p_{1,2}^{3}I).\]
Suppose that $v$ is an eigenvector of $A_1$, which is different from the all-ones vector, and let $\eta$ be the corresponding eigenvalue. Then $Jv = 0$ and $A_1v = \eta v$, so the non-trivial eigenvalue $\eta$ is a root of the polynomial 
$\eta^3+a_1\eta^2+a_2\eta+a_3$.

\end{proof}

\begin{proposition}\label{assoc-spectral-radius}
Fix $\varepsilon>0$. Let $\mathfrak{X}$ be an association scheme of rank 4. Suppose that the parameters of $\mathfrak{X}$ satisfy ${1}/{\varepsilon}\leq k_1$ and $p^{3}_{1,i}\leq \varepsilon k_1$ for $i = 1,2$. Then the zero-weight spectral radius $\xi(X_1)$ of $X_1$ satisfies
\[\xi(X_1)\leq \frac{p_{1,1}^{1}+p_{1,2}^{2}+\sqrt{(p_{1,1}^{1} - p_{1,2}^{2})^2+4p_{1,1}^{2}p_{1,2}^{1}}}{2}+25\varepsilon^{1/3} k_1. \]
\end{proposition}
\begin{proof}
By Lemma \ref{rank-4-roots}, every non-trivial eigenvalue of $X_1$ is a root of the polynomial
\[\eta^3+a_1\eta^2+a_2\eta+a_3, \]
where $a_1$, $a_2$ and $a_3$ are as in Lemma~\ref{rank-4-roots}. Observe, that for
\[ b_1  = -(p_{1,1}^{1}+p_{1,2}^{2}), \quad b_2 = p_{1,2}^2 p_{1,1}^{1} - p_{1,1}^{2} p_{1,2}^{1}, \quad b_3 = 0,\]
the following inequalities are true
\[ |a_1 - b_1|\leq 2\varepsilon k_1, \quad
|a_2 - b_2| \leq \left(4\varepsilon+2\varepsilon^2+\frac{1}{k_1}\right)k_1^2, \quad   |a_3 - b_3|\leq 2k_1^2\leq 2\varepsilon k_1^{3}.\]

Denote by $\nu_1, \nu_2, \nu_3$ the non-trivial eigenvalues of $A_1$. By Theorem \ref{poly-approx}, we can arrange the roots $x_1, x_2, x_3$ of  $x^3+b_1x^2+b_2x+b_3$ so that $|\nu_i - x_i|\leq \delta$, where
\[\delta = 6 \left(2\varepsilon k_1(4k_1)^2+6\varepsilon k_1^2 (4k_1)+2\varepsilon k_1^3\right)^{1/3}\leq 25\varepsilon^{1/3} k_1.\qedhere\]

\end{proof}

\begin{proposition}\label{assoc-x12-approx}
Fix $\varepsilon>0$. Let $\mathfrak{X}$ be an association scheme of rank 4. Suppose that the intersection numbers of $\mathfrak{X}$ satisfy ${1}/{\varepsilon}\leq k_1$, $p_{1,1}^{2}\leq \varepsilon k_1$ and $p^{3}_{i,j}\leq \varepsilon \min(k_{i}, k_{j})$ for $\{i,j\} = \{1,2\}$. Then the zero-weight spectral radius of $X_{1,2}$ satisfies
\[\xi(X_{1, 2})\leq \frac{p_{1,1}^{1}+p_{1,2}^{2}+p_{2,2}^{2}+\sqrt{(p_{2,2}^{2}+ p_{1,2}^{2}- p_{1,1}^{1})^2+4p_{1,2}^{2}p_{2,2}^{1}}}{2}+25\varepsilon^{1/3} (k_1+k_2). \]
\end{proposition}
\begin{notation} We use the non-asymptotic notation $y = \square(x)$ to say that $|y|\leq x$.
\end{notation}
\begin{proof}
The proof is similar to the proofs of Proposition \ref{assoc-spectral-radius} and Lemma \ref{rank-4-roots}.
Denote $k = k_1+k_2$.  By Eq.~\eqref{eq-int-num} we have
\begin{equation}
(A_1+A_2)^{2} = (p_{1,1}^{1}+2p_{1,2}^{1}+p_{2,2}^{1})A_1+(p_{1,1}^{2}+2p_{1,2}^{2}+p_{2,2}^{2})A_2+(p_{1,1}^{3}+2p_{1,2}^{3}+p_{2,2}^{3})A_3+kI.
\end{equation}
Note that, by assumptions of this proposition 
\[0\leq p_{1,1}^{3}+2p_{1,2}^{3}+p_{2,2}^{3} \leq 2\varepsilon k\quad \text{and} \quad 0\leq p_{1,2}^{1} = \dfrac{k_2}{k_1}p_{1,1}^2\leq \varepsilon k.\]
 Using Eq. \eqref{eq-constit-sum}, we eliminate $A_3$.
\begin{equation}\label{eq-eigx12-sq}
\begin{multlined}
(A_1+A_2)^{2} = (p_{1,1}^{1}+p_{2,2}^{1}+2\square(\varepsilon k))A_1+ (2p_{1,2}^{2}+p_{2,2}^{2}+2\square(\varepsilon k))A_2+
\\
+(k+2\square(\varepsilon k))I+2\square(\varepsilon k)J = \\
 = (p_{1,1}^{1}+p_{2,2}^{1}+2\square(\varepsilon k))(A_1+A_2)+2\square(\varepsilon k^2)I+2\square(\varepsilon k)J+\\
 +\left(2p_{1,2}^{2}+p_{2,2}^{2}-p_{1,1}^{1}-p_{2,2}^{1}+4\square(\varepsilon k)\right)A_2.
\end{multlined}
\end{equation}
Denote by $R = 2p_{1,2}^{2}+p_{2,2}^{2}-p_{1,1}^{1}-p_{2,2}^{1}+4\square(\varepsilon k)$ the last coefficient in Eq.~\eqref{eq-eigx12-sq}. Multiplying Eq. \eqref{eq-eigx12-sq} by $(A_1+A_2)$ we get
\begin{equation}\label{eq-eigx12-cub}
\begin{multlined}
(A_1+A_2)^{3} = (p_{1,1}^{1}+p_{2,2}^{1}+2\square(\varepsilon k))(A_1+A_2)^{2}+2\square(\varepsilon k^2)(A_1+A_2)+2\square(\varepsilon k^2)J+\\
+R\left(  (p_{1,2}^{1}+p_{2,2}^{1}+\square(\varepsilon k))(A_1+A_2)+(k_2+\square(\varepsilon k))I+\square(\varepsilon k)J \right)+\\
+R(p_{1,2}^{2}+p_{2,2}^{2}-p_{1,2}^{1}+p_{2,2}^{1})A_2 = \\
 = (p_{1,1}^{1}+p_{2,2}^{1}+2\square(\varepsilon k))(A_1+A_2)^{2}+9\square(\varepsilon k^2)(A_1+A_2)+5\square(\varepsilon k^2)J+\\
 + (2p_{1,2}^{2}+p_{2,2}^{2}-p_{1,1}^{1}-p_{2,2}^{1})p_{2,2}^{1}(A_1+A_2)+3\square(\varepsilon k^{2})I\\
 +R(p_{1,2}^{2}+p_{2,2}^{2} - p_{1,2}^{1} - p_{2,2}^{1})A_2.
\end{multlined}
\end{equation}

Let us multiply Eq. \eqref{eq-eigx12-sq} by $p_{1,2}^{2}+p_{2,2}^{2} - p_{1,2}^{1} - p_{2,2}^{1} = p_{1,2}^{2}+p_{2,2}^{2} - p_{2,2}^{1}+\square(\varepsilon k)$ to eliminate $A_2$ from Eq.~\eqref{eq-eigx12-cub}. Observe first, that 
\[(2p_{1,2}^{2}+p_{2,2}^{2} - p_{1,1}^{1} - p_{2,2}^{1})p_{2,2}^{1} - (p_{1,1}^{1}+p_{2,2}^{1})(p_{1,2}^{2}+p_{2,2}^{2} - p_{2,2}^{1}) = p_{1,2}^{2}p_{2,2}^{1} - p_{1,1}^{1}p_{1,2}^{2} - p_{1,1}^{1}p_{2,2}^{2}.\]
Thus, 
\begin{equation}\label{eq-poly-for-x12}
\begin{multlined}
(A_1+A_2)^{3} - (p_{1,1}^{1}+p_{2,2}^{2} +p_{1,2}^{2})(A_1+A_2)^{2}-\\-(p_{1,2}^{2}p_{2,2}^{1} - p_{1,1}^{1}p_{1,2}^{2} - p_{1,1}^{1}p_{2,2}^{2})(A_1+A_2)+\\
+3\square(\varepsilon k)(A_1+A_2)^{2}+13\square(\varepsilon k^{2})(A_1+A_2)+5\square(\varepsilon k^{3})I+8\square(\varepsilon k^{2})J = 0.
\end{multlined}
\end{equation}
Consider
\[ b_1 = -(p_{1,1}^{1}+p_{2,2}^{2} +p_{1,2}^{2}), \quad b_2 =  p_{1,1}^{1}(p_{1,2}^{2} + p_{2,2}^{2}) - p_{1,2}^{2}p_{2,2}^{1}, \quad b_3 = 0.\]
Then, Eq. \eqref{eq-poly-for-x12} implies that every non-trivial eigenvalue $\eta$ of $X_{1,2}$ satisfies the polynomial equation $\eta^{3}+a_1\eta^{2}+a_2\eta+a_3 = 0$, where
\[ |a_1 - b_1|\leq 3\varepsilon k, \quad |a_2 - b_2|\leq 13\varepsilon k^2, \quad |a_3 - b_3|\leq 5\varepsilon k^{3}.\]

Denote by $\nu_1, \nu_2, \nu_3$ the non-trivial eigenvalues of $A_1+A_2$. By Theorem \ref{poly-approx}, we can permute the roots $x_1, x_2, x_3$ of  $x^3+b_1x^2+b_2x+b_3$ so that $|\nu_i - x_i|\leq \delta$, where
\[\delta = 6 \left(3\varepsilon k(2k)^2+13\varepsilon k^2 (2k)+5\varepsilon k^3\right)^{1/3}\leq 25\varepsilon^{1/3} k.\]
Here we use that the inequalities $b_{1} = p_{1,1}^{1}+(p_{2,1}^{2}+p_{2,2}^{2})\leq k$ and $b_2\leq k^{2}$ hold, by Eq. \eqref{eq-sum-param}.

\end{proof}

\section{Reduction to the case of a constituent with a clique geometry}\label{sec-coh-reduction} 

In this section we show that the motion of a rank-4 association scheme of diameter 2 is linear in the number of vertices, unless one of its constituents, or its complement, has a clique geometry.

First, we show that one can assume that some intersection numbers are small. 

\begin{lemma}\label{assoc-k2-large}
Let $\mathfrak{X}$ be an association scheme of rank 4 and diameter $2$ with the constituents ordered by degree. If $k_2\geq \gamma k_3$, then every pair of distinct vertices is distinguished by at least $\gamma n/6$ vertices.
\end{lemma}
\begin{proof}
Since $\mathfrak{X}$ has diameter 2 it is enough to show that some pair of vertices is distinguished by at least ${\gamma n}/{3}$ vertices, as then result follows by Lemma \ref{babai-dist}. Observe that vertices $u,v$, connected by an edge of color $i$, are distiguished by at least $|N_2(u)\bigtriangleup N_2(v)| = 2(k_2 - p_{2,2}^i)$ vertices. At the same time, we have
\[k_2 (k_2 - 1) = \sum\limits_{i = 1}^{3} k_i p_{2, 2}^{i}\geq k_2 p_{2, 2}^2+k_3 p_{2, 2}^3.\]
Thus, $k_3\geq k_2$ implies $k_2-1\geq p_{2,2}^{2}+p_{2,2}^{3}$. So $\min(p_{2,2}^{2}, p_{2,2}^{3})\leq (k_2-1)/{2}$. Hence, a pair of vertices connected by an edge of color $i$, which minimizes $p_{2,2}^{i}$, is distinguished by at least $k_2+1\geq \gamma k_3+1\geq {\gamma n}/{3}$ vertices.
\end{proof}
\begin{remark}
Note that the result of the lemma can also be derived directly from Proposition~6.3 proven by Babai in \cite{Babai-annals}.
\end{remark}

\begin{lemma}\label{assoc-param-ineq}
Let $\mathfrak{X}$ be an association scheme of rank 4 and diameter $2$ with the largest degree equal $k_3$. Fix some $\varepsilon>0$. 
Assume $\max(k_1, k_2)\leq \varepsilon k_3/2$. Then
\begin{equation}\label{eq-statement-2} 
\quad p_{1,2}^{3}\leq \varepsilon k_1, \quad p_{1,1}^{3}\leq \varepsilon k_1, \quad p_{2,2}^{3}\leq \varepsilon k_2, \quad \text{and}
\end{equation}
\begin{equation}
p^{1}_{3,3}\geq k_3(1-\varepsilon), \quad p^{2}_{3,3}\geq k_3(1-\varepsilon).
\end{equation}

\end{lemma}
\begin{proof}
 Note that for $i=1,2$,
 \[k_i(k_i-p_{i,i}^{i}-1)\geq k_3p^{3}_{i,i},\quad \text{so}\quad 
 p^{3}_{i,i}\leq \varepsilon k_i/2.\]
 Additionally, $p^{1}_{2,3}\leq k_2\leq \varepsilon k_3/2$. Thus, by Eq.~\eqref{eq-sum-param}, $p_{1,2}^{3} = k_1 p_{2,3}^{1}/ k_3 \leq \varepsilon k_1/2$.

Finally, by Eq.~\eqref{eq-sum-param}, for $j\in \{1, 2\}$, we have \[p_{3,3}^i+p_{3, 2}^i+p_{3, 1}^i = k_3\quad \text{and} \quad p_{3,j}^{i}\leq k_j\leq \varepsilon k_3/2. \] 
Therefore, $k_{3,3}^{i}\geq (1-\varepsilon)k_3$.

\end{proof}
\begin{remark} Note that the inequalities in Eq. \eqref{eq-statement-2} are still true if we replace $\varepsilon$  by $\varepsilon/2$.
\end{remark}

We need the following lemma, corollaries of which will be used several times.

\begin{lemma}\label{assoc-tr} Let $\mathfrak{X}$ be an association scheme. Suppose that there exists a triangle with sides of colors $(s,r,t)$. Then  
\[  p_{i,j}^{s}+p_{j,l}^{r}\leq k_j+p_{i,l}^{t}. \]
\end{lemma}
\begin{proof}
Apply the inclusion
\[N_i(u)\setminus N_l(w) \subseteq (N_i (u)\setminus N_j(v))\cup (N_j(v)\setminus N_l(w)),\] to vertices $u,v,w$, where $c(u,v) = s$, $c(v,w) = r$, and $c(u,w) = t$.
\end{proof}
\begin{corollary}\label{cor1} Suppose that an association scheme $\mathfrak{X}$ of rank 4 satisfies $\max(k_1, k_2)\leq \varepsilon k_3/2$. Suppose also that there exists a triangle with sides $(s,t,3)$.
Then,  $p_{i,j}^{s}\leq p^{t}_{i,3}+\varepsilon k_j$, where $i,j\in \{1,2\}$.
\end{corollary}
\begin{proof}
Take $r = l = 3$ in Lemma \ref{assoc-tr}, then by Lemma~\ref{assoc-param-ineq}, $p_{j,3}^{3} = {k_j}p_{3,3}^{j}/{k_3}\geq (1-\varepsilon)k_j$.
\end{proof}

\begin{corollary}\label{cor2} Suppose that an association scheme $\mathfrak{X}$ has rank 4 and diameter~2. Moreover, assume $\max(k_1, k_2)\leq {\varepsilon} k_3/2$.
Then,  $p_{i,j}^{s}\leq p^{s}_{i,3}+\varepsilon k_j$, where $i,j,s \in \{1,2\}$.
\end{corollary}
\begin{proof}
Take $r = l = 3$ and $s = t$ in Lemma \ref{assoc-tr}. Observe, that a triangle with sides of colors $(s,s,3)$ exists, as diameter of $\mathfrak{X}$ is 2 and $s\neq 3$.
\end{proof}

\begin{corollary}\label{cor3} Suppose an association scheme $\mathfrak{X}$ has rank 4 and diameter~2. Assume $\max(k_1, k_2)\leq \varepsilon k_3/2$.
Then, $2p_{i,j}^{s}\leq k_j+\varepsilon k_i$, where $i,j,s\in \{1,2\}$. Moreover, if $k_1\leq k_2$, then $2p_{1,2}^{2}\leq (1+\varepsilon)k_1$.
\end{corollary}
\begin{proof}
 Take $t = 3$, $s = r$ and $l = i$ in Lemma \ref{assoc-tr} and we use Lemma \ref{assoc-param-ineq}. Take $s=1$ and $i=j=2$, then $2p_{2,2}^{1}\leq (1+\varepsilon)k_2$, so $2p_{1,2}^{2} \leq \dfrac{k_1}{k_2} (1+\varepsilon)k_2 = (1+\varepsilon)k_1$.
\end{proof}

We state the following simple corollary of Metsch's criteria (Theorem \ref{Metsch}) and of the classification of graphs with the smallest eigenvalue $\geq -2$ (Theorems \ref{geom-eig-more2} and \ref{geom-eig-2}), which will be used in the proof of Theorem~\ref{assoc-diam2}. 

\begin{lemma}\label{assoc-clique-geom}
Let $\mathfrak{X}$ be an association scheme of rank $r\geq 4$ and diameter 2 on $n$ vertices.
\begin{enumerate}
 \item Assume that for some $i$ the constituent $X_i$ satisfies the assumptions of Theorem \ref{Metsch} for $m=2$. Then $X_i$ is a strongly regular graph with smallest eigenvalue $-2$, or is the line graph of a regular triangle-free graph.
 
 Furthermore, $X_i$ satisfies the assumptions of Theorem \ref{Metsch} for $m=2$, if we have one of the following
 \begin{enumerate}
 \item $\lambda(X_i) = p_{i,i}^{i}\geq\dfrac{2}{5}k_i$ and $\mu(X_i) = \max\{p_{i,i}^{j}: 0<j\neq i\} \leq \dfrac{1}{30}k_i$, or
 \item $\lambda(X_i) = p_{i,i}^{i}\geq\left(\dfrac{1}{2}-\dfrac{1}{20}\right)k_i$ and $\mu(X_i) = \max\{p_{i,i}^{j}: 0<j\neq i\} \leq \dfrac{1}{11}\left(1+\dfrac{1}{100}\right)k_i$. 
\end{enumerate}

\item Assume for some $i$ the complement $\overline{X_{i}}$ of $X_i$ satisfies the assumptions of Theorem~\ref{Metsch} for $m=2$. If $n\geq 12$, then graph $\overline{X_{i}}$ is strongly regular with smallest eigenvalue $-2$.  
\end{enumerate} 
\end{lemma}
\begin{proof}
\begin{enumerate}
\item First, we check that $X_i$ satisfies the conditions of Theorem \ref{Metsch} for $m=2$.
It is sufficient to verify that $\lambda(X_i)>3\mu(X_i)$ and $3\lambda(X_i) - 3\mu(X_i)>k_i$. 
\begin{enumerate}
\item We compute 
$$\lambda(X_i) \geq \frac{2}{5}k_i > \frac{3}{30}k_i\geq 3\mu(X_i)\quad \text{and}\quad 3\lambda(X_i)-3\mu(X_i)\geq \left(\frac{6}{5}-\frac{1}{10}\right)k_i>k_i.$$
\item We compute 
\[\lambda(X_i) - 3\mu(X_i) \geq \left(\frac{1}{2} - \frac{1}{20} - \frac{3}{11}\left(1+\frac{1}{100}\right)\right)k_i = \frac{48}{275}k_i>0,\]
\[3\lambda(X_i)-3\mu(X_i)\geq k_i+\left(\frac{1}{2} - \frac{3}{20} - \frac{3}{11}\left(1+\frac{1}{100}\right)\right)k_i = k_i+\frac{41}{550}k_i>k_i.\]
\end{enumerate}
Now, if $X_i$ satisfies the conditions of Theorem~\ref{Metsch} for $m=2$, by Lemma~\ref{m2-line}, it is a line graph, and, by Lemma \ref{geom-smallest-eigenvalue}, the smallest eigenvalue of $X_i$ is at least $-2$. Moreover, recall that $X_i$ is edge-regular. If the smallest eigenvalue is $>-2$, by Theorem~\ref{geom-eig-more2}, $X_i$ is a complete graph or an odd polygon. This is impossible, since $\mathfrak{X}$ has diameter $2$ and at least three non-empty constituents.  If the smallest eigenvalue is $-2$, then by Theorem~\ref{geom-eig-2}, we get that $X_i$ is a strongly regular graph, or is the line graph of a regular triangle-free graph.

\item Since $\overline{X_{i}}$ satisfies the conditions of Theorem~\ref{Metsch}, by Lemmas~\ref{m2-line} and~\ref{geom-smallest-eigenvalue}, graph $\overline{X_{i}}$ is a line graph and its smallest eigenvalue is at least $-2$. Note also that $\overline{X_{i}}$ is co-edge-regular. If the smallest eigenvalue is $>-2$, by Theorem~\ref{geom-eig-more2}, $\overline{X_{i}}$ is complete graph or an odd polygon. This is impossible, since $\mathfrak{X}$ has diameter $2$ and at least three non-empty constituents. If the smallest eigenvalue is $-2$, then by Theorem~\ref{geom-eig-2}, we get that $\overline{X_{i}}$ is strongly regular, an $m_1\times m_2$-grid or one of the two regular subgraphs of the Clebsh graph with 8 or 12 vertices.

Assume $\overline{X_{i}}$ is a $m_1\times m_2$-grid with $m_1\neq m_2$ and $m_1, m_2>1$. That is, $\overline{X_{i}}$ is the line graph of $K_{m_1, m_2}$. Denote the parts of $K_{m_1, m_2}$ by $U_1$ and $U_2$ with $|U_i| = m_i$. By symmetry, we can assume $m_1< m_2$. We compute $n = m_1m_2$, $k_1+k_2 = m_1+m_2 - 2$ and $\mu = 2$. Observe that, two edges of $K_{m_1, m_2}$ that share a vertex in $U_i$ have $m_i - 2$ common neighbors. Since $m_1 \neq m_2$, two pairs of edges in $K_{m_1, m_2}$ that share a vertex in $U_1$ and $U_2$, respectively, cannot be colored in the same color in  $\mathfrak{X}$. Thus, in particular, $\mathfrak{X}$ is not primitive.  

Therefore, $\overline{X_{i}}$ is strongly regular.  
\end{enumerate}
\end{proof}

\begin{theorem}\label{assoc-diam2}
Let $\mathfrak{X}$ be an association scheme of rank 4 on $n$ vertices  with diameter~2 and with constituents ordered by degree. Recall that $q(X_i) = \max\{p_{i,i}^{j}: j\in [3]\}$ is the maximal number of common neighbours of two distinct vertices in $X_i$. Fix $\varepsilon = 10^{-16}$. Then one of the following is true.
\begin{enumerate}
  \item Every pair of distinct vertices is distinguished by at least $\varepsilon n /12$ vertices.
  \item The zero-weight spectral radius $\xi(X_i)$ of $X_i$ satisfies $q(X_i)+\xi(X_i)\leq (1-\varepsilon)k_i$ for $i = 1$ or $i = 2$.  
  \item The graph $X_1$ is either strongly regular with smallest eigenvalue $-2$, or the line graph of a connected regular triangle-free graph. 
  \item The graph $X_2$ is either strongly regular with smallest eigenvalue $-2$, or the line graph of a connected regular triangle-free graph. Moreover, $k_2\leq \frac{101}{100}k_1$.
  \item If $n\geq 12$, then the graph $X_{1,2}$ is strongly regular with smallest eigenvalue $-2$ and $k_2\leq \frac{101}{100}k_1$.
  \end{enumerate}  
\end{theorem}
\begin{proof}
We may assume that parameters of $\mathfrak{X}$ satisfy $\max(k_1, k_2)\leq {\varepsilon} k_3/2$, as otherwise, by Lemma~\ref{assoc-k2-large} every pair of distinct vertices is distinguished by at least $\varepsilon n/12$ vertices. Therefore, all the inequalities provided by Lemma \ref{assoc-param-ineq} hold. 

Thus, by Proposition \ref{assoc-spectral-radius},
\[\xi(X_1)\leq \frac{p_{1,1}^{1}+p_{1,2}^{2}+\sqrt{(p_{1,1}^{1}-p_{1,2}^{2})^{2}+4p_{1,1}^{2}p_{1,2}^{1}}}{2}+\varepsilon_1 k_1, \text{ so}\]
\begin{equation}\label{x1-eq-bound}
\xi(X_1)\leq \max(p_{1,1}^{1}, p_{1,2}^{2})+\sqrt{p_{1,1}^{2} p_{1,2}^{1}}+\varepsilon_1 k_1,
\end{equation}
where $\varepsilon_1  = 25\varepsilon^{1/3}$. 
Similarly,
\begin{equation}\label{x2-eq-bound}
\xi(X_2) \leq \max(p_{2,2}^{2}, p_{2,1}^{1})+\sqrt{p_{2,2}^{1}p_{1,2}^{2}}+\varepsilon_1 k_2.
\end{equation}

We note that $k_1 \geq {p_{1,1}^3}/{\varepsilon}\geq {1}/{\varepsilon}$, by Eq.~\eqref{eq-statement-2}.

\noindent \textbf{Case 1.} Assume $\gamma k_2> k_1$, where $\gamma = \frac{1}{900}$. Then, using Lemma~\ref{assoc-param-ineq},
 
 \[p_{1,1}^2  = \dfrac{k_1}{k_2} p_{1,2}^{1}\leq \gamma p_{1,2}^{1},\quad \text{so} \quad \mu(X_1)\leq \max(\gamma, \varepsilon) k_1  = \frac{1}{900}k_1.\]
  Note that, by Corollary \ref{cor3}, inequality $\max(p_{1,1}^{1}, p_{1,2}^{2})\leq \dfrac{1+\varepsilon}{2}k_1$ holds. Hence, by Eq.~\eqref{x1-eq-bound},
\begin{equation}\label{eq-case-1}
q(X_1)+\xi(X_1)\leq ((\varepsilon+ \gamma)k_1+p_{1,1}^{1})+\left(\frac{1+\varepsilon}{2}k_1+\sqrt{\gamma}k_1+\varepsilon_1 k_1\right).
\end{equation}

If $p_{1,1}^{1}<\dfrac{2}{5}k_1$, then Eq. \eqref{eq-case-1} implies $q(X_1)+\xi(X_1)\leq (1-\varepsilon)k_1$. Otherwise, if $p_{1,1}^{1}\geq\dfrac{2}{5}k_1$, by Lemma \ref{assoc-clique-geom}, the graph $X_1$ satisfies the statement 3 of this proposition.

\noindent \textbf{Case 2.} Assume $k_1 = \gamma k_2$, where $(1+\varepsilon_3)^{-1} \geq \gamma\geq \frac{1}{900}$, for $\varepsilon_3 = \frac{1}{100}$. 

We consider two subcases.

\noindent\textbf{Case 2.1.}
Suppose $p_{2,3}^{1} = 0$. 

So, by Corollary \ref{cor2},
\begin{equation}
p_{2,2}^{1}\leq \varepsilon k_2+p_{2,3}^{1} = \frac{\varepsilon}{\gamma}k_1, \quad p_{1,2}^{2} = \frac{k_1}{k_2} p_{2,2}^{1} \leq \varepsilon k_1,
\end{equation}
\begin{equation}\label{eq212}
\quad p_{1,2}^{1}\leq \varepsilon k_1 + p_{2,3}^{1} = \varepsilon k_1, \quad p_{1,1}^{2} = \frac{k_1}{k_2} p_{1,2}^{1} \leq \varepsilon k_1.
\end{equation} 
 Then,
\begin{equation}\label{x1-eq-bound-neeew}
 \max_{i\in [3]}(p_{1,1}^{i})+\max(p_{1,1}^{1}, p_{1,2}^{2})+\sqrt{p_{1,1}^{2} p_{1,2}^{1}} \leq 3\varepsilon k_1 +2p_{1,1}^{1}.
\end{equation} 
 Thus, by Eq.~\eqref{x1-eq-bound-neeew}, $q(X_1)+\xi(X_1)<(1-\varepsilon)k_1$  if $p_{1,1}^{1}<\dfrac{2}{5}k_1$. Alternatively, if $p_{1,1}^{1}\geq \dfrac{2}{5}k_1$, by Lemma \ref{assoc-clique-geom}, graph $X_1$ satisfies the statement 3 of this proposition, as $\mu(X_1)\leq \varepsilon k_1$ by Eq.~\eqref{eq212}.

\noindent\textbf{Case 2.2.}
 Suppose $p_{2,3}^{1} \neq 0$. 
\begin{enumerate}
\item[]\textbf{Case 2.2.1.} Assume that $p_{1,1}^{1}\geq p_{1,1}^{2}$. 

By Corollary \ref{cor1},
 $$p_{1,2}^{2}\leq p_{1,3}^{1}+\varepsilon k_2 = p_{1,3}^{1}+\frac{\varepsilon}{\gamma}k_1 \quad \text{and} \quad p_{1,1}^{1}\leq p_{1,3}^{1}+\varepsilon k_1.$$
Then 
\begin{equation}\label{eq-case-2211}
\max_{i\in [3]}(p_{1,1}^{i})+p_{1,2}^{2}+\sqrt{\gamma}p_{1,2}^{1}\leq 
p_{1,1}^{1}+p_{1,3}^{1}+p_{1,2}^{1} - (1-\sqrt{\gamma})p_{1,2}^{1}+\left(\varepsilon+\frac{\varepsilon}{\gamma}\right)k_1,
\end{equation}
\begin{equation}
\max_{i\in [3]}(p_{1,1}^{i})+p_{1,1}^{1}+\sqrt{\gamma}p_{1,2}^{1}\leq 
p_{1,3}^{1}+p_{1,1}^{1}+p_{1,2}^{1} - (1-\sqrt{\gamma})p_{1,2}^{1}+\left(\varepsilon+\frac{\varepsilon}{\gamma}\right)k_1 .
\end{equation}

\item[]\textbf{Case 2.2.2.} Assume that $p_{1,1}^{2}\geq p_{1,1}^{1}$. 

By Corollary \ref{cor2}, 
\[p_{1,1}^{2} = \gamma p_{1,2}^{1}\leq \gamma (p_{1,3}^{1}+\varepsilon k_2) \leq \gamma p_{1,3}^{1}+\varepsilon k_1. \]
This implies
\begin{equation}
 \max_{i\in [3]}(p_{1,1}^{i})+p_{1,1}^{1}+\sqrt{\gamma}p_{1,2}^{1}\leq \gamma p_{1,3}^{1}+p_{1,1}^{1}+\sqrt{\gamma}p_{1,2}^{1}+\varepsilon k_1,
\end{equation}
\begin{equation}\label{eq-case-2222}
 \max_{i\in [3]}(p_{1,1}^{i})+p_{1,2}^{2}+\sqrt{\gamma}p_{1,2}^{1}\leq p^{2}_{1, 1}+p_{1,2}^{2}+p_{1,3}^{2} - (1-\sqrt{\gamma})p_{1,2}^{1} +\left(\varepsilon +\frac{\varepsilon}{\gamma}\right)k_1,
\end{equation}
where in Eq.~\eqref{eq-case-2222} we use the inequality $p_{1,2}^{1}\leq p_{1,3}^{2}+\varepsilon k_2$ given by Corollary \ref{cor1}.
\end{enumerate}
Therefore, using Eq. \eqref{eq-sum-param}, in both subcases by Eq.~\eqref{eq-case-2211}~-~\eqref{eq-case-2222} we get
\[ \max_{i\in [3]}(p_{1,1}^{i})+\left(\max(p_{1,1}^{1}, p_{1,2}^{2})+\sqrt{p_{1,1}^{2} p_{1,2}^{1}}\right) \leq k_1 - (1-\sqrt{\gamma})p_{1,2}^{1} +\left(\varepsilon +\frac{\varepsilon}{\gamma}\right)k_1, \text{ so}\]
\begin{equation}\label{eq-two-case-bound}
q(X_1)+\xi(X_1)\leq k_1 - (1-\sqrt{\gamma})p_{1,2}^{1} +\left(\varepsilon +\frac{\varepsilon}{\gamma}\right)k_1 +\varepsilon_1 k_1.
\end{equation}
We again consider two subcases.

\begin{enumerate}
\item[] \textbf{Case 2.2.a.} Suppose $p_{1,2}^{1}>\varepsilon_2 k_1$ for $\varepsilon_2 = \frac{1}{30}$.

Observe that
\[\varepsilon_2(1-\sqrt{\gamma}) - \varepsilon\left(2+\frac{1}{\gamma}\right)-\varepsilon_1\geq 10^{-4} - 902\varepsilon - 25\varepsilon^{1/3}>0,\]
 so by Eq.~\eqref{eq-two-case-bound},
\begin{equation}\label{eq-e2-small}
q(X_1)+\xi(X_1)  \leq (1-\varepsilon)k_1.
\end{equation}

\item[]\textbf{Case 2.2.b.} Suppose $p_{1,2}^{1}\leq \varepsilon_2 k_1$. 

This implies $p_{1,1}^{2}\leq \varepsilon_2 k_1$, so $\mu(X_1)\leq \frac{1}{30}k_1$. Recall, that by Corollary~\ref{cor3}, the  inequality $\max(p_{1,2}^{2}, p_{1,1}^{1})\leq \dfrac{1+\varepsilon}{2} k_1 $ holds. Then, by Eq~\eqref{x1-eq-bound},  we have 
\begin{equation}\label{eq-e2-large}
q(X_1)+\xi(X_1)\leq (\varepsilon+\varepsilon_2)k_1+p_{1,1}^{1}+\frac{1+\varepsilon}{2}k_1+ \varepsilon_2k_1+\varepsilon_1 k_1.
\end{equation} 
 Thus, either Eq.~\eqref{eq-e2-large} implies the inequality $q(X_1)+\xi(X_1)<(1-\varepsilon)k_1$, or $p_{1,1}^{1}\geq \dfrac{2}{5} k_1$. In the latter case, by Lemma \ref{assoc-clique-geom},  the statement 3 of this proposition holds for $X_1$.
 \end{enumerate}

\noindent \textbf{Case 3.} Suppose that $k_2\leq (1+\varepsilon_3)k_1$, where $\varepsilon_3 = \frac{1}{100}$. 

\noindent In this case, we work with both $X_1$ and $X_2$ in the same way. Additionally, we need to consider the graph $X_{1,2}$ with the set of vertices $V(X_{1,2}) = V(X_1) = V(X_2)$ and set of edges $E(X_{1,2}) = E(X_1)\cup E(X_2)$. The graph $X_{1,2}$ is regular of degree $k_1+k_2$, and every pair of non-adjacent vertices has 
\begin{equation}\label{eq-mux12-above}
\mu(X_{1,2}) = p_{1,1}^{3}+2 p_{1,2}^{3}+p_{2,2}^{3} \leq 2\varepsilon (k_1+k_2) \leq 4\varepsilon(1+\varepsilon_3)k_1 
\end{equation} 
common neighbors. Every pair of vertices connected by an edge of color $i$ has 
\begin{equation}\label{eq:lambda-i}
\lambda_i = p_{1,1}^{i}+2 p_{1,2}^{i}+p_{2,2}^{i}
\end{equation} common neighbors, for $i = 1,2$. We apply  the inequality 
\begin{equation}
 |N(u)\cap N(v)|+|N(v)\cap N(w)|\leq |N(v)|+|N(u)\cap N(w)|
\end{equation}
  to the graph $X_{1, 2}$ and vertices $u$, $v$, $w$ with $c(u, v) = c(v, w) = i$ and $c(u, w) = 3$, where $i\in \{1,2\}$. We get $2\lambda_i \leq k_1+k_2+\mu(X_{1,2})$, so by Eq.~\eqref{eq-mux12-above},
\begin{equation}\label{eq-lambda-above}
\lambda_i = p_{1,1}^{i}+2 p_{1,2}^{i}+p_{2,2}^{i}\leq \frac{1+2\varepsilon}{2}(k_1+k_2) \leq k_1(1+\varepsilon_3+2\varepsilon).
\end{equation} 
Let $\{i, j\} = \{1,2\}$, then by Eq.~\eqref{x1-eq-bound}-\eqref{x2-eq-bound},
\begin{equation}\label{eq-case-3}
\begin{aligned}
q(X_i)+\xi(X_i) & \leq q(X_i)+\max( p_{i,i}^{i}, p_{i,j}^{j})+\sqrt{p_{i,j}^{i}p_{i,i}^{j}} +\varepsilon_1 k_i \leq  \\ & \leq \left(\max( p_{i,i}^{i}, p_{i,i}^{j})+\varepsilon k_i\right)+ \max( p_{i,i}^{i}, p_{i,j}^{j})+p_{i,j}^{i}(1+\varepsilon_3)+\varepsilon_1 k_i.
\end{aligned}
\end{equation}
Consider all possible ways of opening the maximums in Eq.~\eqref{eq-case-3} (we only write terms without epsilons).
\begin{enumerate}
\item $2p_{i,i}^{i}+p_{i,j}^{i}$,
\item $p_{i,i}^{i}+p_{i,i}^{j}+p_{i,j}^{i}\leq (1+\varepsilon_3)(p_{i,i}^{i}+2p_{i,j}^{i}) = (1+\varepsilon_3)(\lambda_i - p_{j,j}^{i})\leq (1+\varepsilon_3)\lambda_i,$
\item $p_{i,j}^{j}+p_{i,i}^{i}+p_{i,j}^{i}\leq (1+\varepsilon_3)(p_{j,j}^{i}+p_{i,i}^{i}+p_{i,j}^{i}) = (1+\varepsilon_3)(\lambda_i - p_{i,j}^{i})\leq (1+\varepsilon_3)\lambda_i,$
\item $p_{i,j}^{j}+p_{i,i}^{j}+p_{i,j}^{i}\leq (1+\varepsilon_3)(p_{j,j}^{i}+2p_{i,j}^{i}) = (1+\varepsilon_3)(\lambda_i - p_{i,i}^{i})\leq (1+\varepsilon_3)\lambda_i.$
\end{enumerate}
Hence, by Corollary \ref{cor3} applied to $p_{i, j}^{i}$, Eq.~\eqref{eq-case-3} implies
\begin{equation}\label{eq-max-open}
q(X_i)+\xi(X_i)\leq \max(2p_{i,i}^{i}+p_{i,j}^{i}, (1+\varepsilon_3)\lambda_i)+k_i\left(\varepsilon_1+\frac{2}{3}\varepsilon_3+\varepsilon\right).
\end{equation}

\noindent \textbf{Case 3.1} Suppose $\lambda_t\geq (2/3+\frac{1}{300})k_t$ for both $t =1,2$.

\noindent Then in notation of Theorem~\ref{Metsch} 
\[\lambda^{(1)} \geq \left(\frac{2}{3}+\frac{1}{300}\right)k_1,\quad \text{and by Eq. \eqref{eq-lambda-above},}\quad \lambda^{(2)} \leq \frac{11}{10}k_1.\]
We check that
\[2\lambda^{(1)} - \lambda^{(2)}\geq 20\varepsilon(1+\varepsilon_3)k_1\geq 5\mu, \quad \text{and} \quad 3\lambda^{(1)} - 3\mu \geq 2k_1+\frac{1}{100}k_1\geq k_1+k_2.\] Thus, $X_{1,2}$ satisfies conditions of Theorem~\ref{Metsch} for $m=2$, so  the statement 5 of this proposition holds by Lemma \ref{assoc-clique-geom}.

\noindent \textbf{Case 3.1} Suppose that $\lambda_i\leq (2/3+\frac{1}{300})k_i$ for some $i\in \{1, 2\}$. 

\noindent If $2p_{i,i}^{i}+p_{i,j}^{i}\leq k_i - k_i(\varepsilon_3 +2\varepsilon+\varepsilon_1)$, then Eq.~\eqref{eq-max-open} implies
 \[q(X_i)+\xi(X_i)\leq (1-\varepsilon)k_i.\]
Hence, we can assume $2p_{i,i}^{i}+p_{i,j}^{i}\geq k_i - k_i(\varepsilon_3 +2\varepsilon+\varepsilon_1)$. Recall, that 
\[ 2p_{i,i}^{i}+p_{i,j}^{i} = \lambda_i+(p_{i,i}^{i} - p_{i,j}^{i}- p_{j,j}^{i})\leq \lambda_i+\frac{1+\varepsilon}{2}k_i - (p_{i,j}^{i}+ p_{j,j}^{i}) .\]
Thus,
\begin{equation}\label{eq:32-new}
 p_{i,j}^{i}+ p_{j,j}^{i}\leq \lambda_i+\frac{1+\varepsilon}{2}k_i-k_i(1-\varepsilon_3 - 2\varepsilon - \varepsilon_1)\leq k_i \frac{51}{300} + k_i (\varepsilon_3+3\varepsilon+\varepsilon_1)\leq \frac{2}{11}k_i.
 \end{equation}
This implies, 
\begin{equation}
\min(p_{j,j}^{i}, p_{i,i}^{j}) \leq (1+\varepsilon_3)\min(p_{i,j}^{i}, p_{j,j}^{i})\leq \frac{1+\varepsilon_3}{11}k_1.
\end{equation}
Take $\{s,t\} = \{1,2\}$, so that $p_{s,s}^{t}\leq \dfrac{1+\varepsilon_3}{11}k_1$. Then 
\[\mu(X_s)\leq \max\left(\varepsilon k_s, \dfrac{1+\varepsilon_3}{11}k_1\right)=\dfrac{1+\varepsilon_3}{11}k_1.\]
 
 We consider two possibilities. 
First, assume that $p_{s,s}^{s}\geq \left(\dfrac{1}{2} - \dfrac{1}{20}\right)k_s$. 
Then, by Lemma~\ref{assoc-clique-geom} graph $X_s$ satisfies the statement 3 or 4 of this proposition.

Assume now that $p_{s,s}^{s}\leq \left(\dfrac{1}{2}-\dfrac{1}{20}\right)k_s$, then
\begin{equation}\label{eq31}
2p_{s,s}^{s}+p_{s,t}^{s}\leq k_s - \frac{1}{10}k_s +\frac{(1+\varepsilon_3)^2}{11}k_s \leq \left(1-2\varepsilon-\frac{2}{3}\varepsilon_3-\varepsilon_1\right)k_s,
\end{equation}
and by Eq.~\eqref{eq:32-new},
\begin{equation}\label{eq32}
\begin{gathered}
(1+\varepsilon_3)\lambda_s\leq p_{s,s}^{s}+2p_{s,t}^{s}+2p_{t,t}^{s}+2\varepsilon_3k_s \leq \\
\\ \leq \left(\frac{1}{2}-\frac{1}{20}+\frac{4+4\varepsilon_3}{11}+2\varepsilon_3\right)k_s<\left(1-2\varepsilon-\frac{2}{3}\varepsilon_3-\varepsilon_1\right)k_s.
\end{gathered}
\end{equation}
Thus, by Eq. \eqref{eq-max-open}, equations \eqref{eq31} and \eqref{eq32} imply
\[q(X_s)+\xi(X_s)\leq  \max(2p_{s,s}^{s}+p_{s,t}^{s}, (1+\varepsilon_3)\lambda_s)+\left(\varepsilon+\frac{2}{3}\varepsilon_3+\varepsilon_1\right)k_s\leq (1-\varepsilon)k_s.\]

\end{proof}

\section{Case of a constituent with a clique geometry}\label{sec-subsec-str-reg}

In the previous subsection, Theorem~\ref{assoc-diam2} reduces the diameter 2 case of Theorem \ref{main-coherent} to the case when one of the constituents is a strongly regular graph with smallest eigenvalue $-2$, or is the line graph of a triangle-free regular graph. In this subsection we resolve the remaining cases.
 
In the case when the dominant constituent $X_3$ is strongly regular we introduce an additional tool (Lemma~\ref{sun-wilmes-tool}), which allows us to bound the order of the group and its minimal degree, when vertices inside a clique are well-distinguished.

In the cases when the constituent $X_1$ or $X_2$ is strongly regular, we prove upper bounds on the quantity $q(X_J)+\xi(X_J)$ for $J\in \{1, 2, \{1, 2\}\}$ with the consequence that the spectral tool (Lemma \ref{mixing-lemma-tool}) can be applied effectively. 

The hardest case in our analysis is the case when the constituent  with the smallest degree, $X_1\,$, is strongly regular. This case is settled in Theorem~\ref{assoc-x1-strongly-reg} (Sec. \ref{sec-x1}) and it requires considerable preparatory work to establish a constant upper bound on the quotient $k_2/k_1$ in certain range of parameters.

\subsection{Triangular graph with well-distinguished cliques}

In the case when the union of some constituents of a homogeneous coherent configuration is a triangular graph we prove the following statement inspired by Lemma 3.5 in \cite{Sun-Wilmes}.

\begin{lemma}\label{sun-wilmes-tool}
Let $\mathfrak{X}$ be a homogeneous coherent configuration on $n$ vertices. Let $I$ be a set of colors, such that if $i\in I$, then $i^{*}\in I$. Suppose that graph $X_I$ is the triangular graph $T(s)$ for some $s$. Let $\mathcal{C}$ be a Delsarte clique geometry in $X_I$. Assume there exists a constant $0<\alpha<\frac{1}{2}$, such that for every clique $C\in \mathcal{C}$ and every pair of distinct vertices $x, y \in C$ there exist at least $\alpha |C|$ elements $z\in C$ which distinguish $x$ and $y$, i.e., $c(z,x) \neq c(z,y)$. Then
 \begin{enumerate}
 \item There exists a set of vertices of size $O(\log(n))$ that completely splits $\mathfrak{X}$. Hence, $|\Aut(\mathfrak{X})| = n^{O(\log(n))}$,
 \item $\motion(\mathfrak{X})\geq \alpha n/2$.
 \end{enumerate}

\end{lemma}
\begin{proof} Consider a clique $C\in \mathcal{C}$. Since every pair of distinct vertices $x,y\in C$ is distinguished by at least $\alpha |C|$ vertices of $C$, by Lemma \ref{Babai-disting-order-group}, there is a set of size at most $\dfrac{2}{\alpha}\log(|C|)+1$ that splits $C$ completely.

Take any vertex $x\in \mathfrak{X}$. By the assumptions of the lemma, $\{x\}\cup N_I(x) = C_1\cup C_2$ for some $C_1, C_2\in \mathcal{C}$. Then there exists a set $S$ of size $\dfrac{4}{\alpha}\log(|C|)+2\leq \dfrac{4}{\alpha}\log(n)+2$, that splits both $C_1$ and $C_2$ completely. Note that every clique $C\in \mathcal{C}$, distinct from $C_1$ and $C_2$, intersects each of them in exactly one vertex, and is uniquely determined by $C\cap C_1$ and $C\cap C_2$. Therefore, the pointwise stabilizer $\Aut(\mathfrak{X})_{(S)}$ fixes every clique $C\in \mathcal{C}$ as a set. At the same time, every vertex $v$ is uniquely defined by the collection of cliques in $\mathcal{C}$ that contain~$v$. Therefore, $S$ splits $X$ completely. 

Suppose $\sigma \in \Aut(\mathfrak{X})$ and $|\supp(\sigma)|< \dfrac{\alpha}{2} n$. Then, by the pigeonhole principle, there exists a vertex $x$, such that $\sigma$ fixes at least $\left\lceil\left(1 - \dfrac{\alpha}{2}\right)(|N_{I}(x)|+1)\right\rceil$ vertices in $N_I(x)\cup \{x\}$. Since $X_I = T(s)$, we have $\{x\}\cup N_I(x) = C_1\cup C_2$ for some $C_1, C_2\in \mathcal{C}$ and $|C_1| = |C_2| = 1+\dfrac{|N_I|}{2}$. Thus $\sigma$ fixes more than $(1-\alpha)|C_i|$ vertices in $C_i$ for every $i\in \{1,2\}$. This means that any pair of vertices  $x, y\in C_i$ is distinguished by at least one vertex fixed by $\sigma$. Hence, $\sigma(x)\neq y$. At the same time, since $(1-\alpha)|C_i|>1$, $\sigma(x)\in C_i$ for every $x\in C_i$. Therefore, $\sigma$ fixes poinwise both $C_1$ and $C_2$. Finally, by the argument in the previous paragraph, we get that $\sigma$ fixes every vertex, so $\sigma$ is the identity.

\end{proof}

\subsection{Constituent $X_3$ is strongly regular} 

In the following theorem we consider the case when the constituent $X_3$ is a strongly regular graph (case of the statement 5 in Theorem~\ref{assoc-diam2}).
 
\begin{theorem}\label{assoc-x3-strongly-reg}
Let $\mathfrak{X}$ be an association scheme of rank 4 and diameter~2 on $n\geq29$ vertices. Assume that the constituents of $\mathfrak{X}$ are ordered by degree and $k_2\leq {\varepsilon}k_3/2$ holds for $\varepsilon < \frac{1}{100} $. Suppose that $k_2\leq \frac{11}{10}k_1$ and $X_{1,2}$ is strongly regular with smallest eigenvalue~$ -2$. Then neither $X_1$, nor $X_2$ is strongly regular with smallest eigenvalue $-2$, and one of the following is true.
\begin{enumerate}
\item The association scheme $\mathfrak{X}$ satisfies the assumptions of Lemma \ref{sun-wilmes-tool} for $I = \{1, 2\}$ and  $\alpha  = 1/{16}$.
\item $X_1$ or $X_2$ is the line graph of a regular triangle-free graph.
\end{enumerate} 
\end{theorem}
\begin{proof}
By the assumptions of the theorem, all inequalities from Lemma \ref{assoc-param-ineq} hold. 

Since $X_{1,2}$ is strongly regular with smallest eigenvalue $-2$, by Seidel's classification (see Theorem \ref{geom-eig-2}), $X_{1,2} = T(s)$ or $X_{1,2} = L_{2}(s)$ for some $s$.  Suppose that $X_{1,2}$ is $L_{2}(s)$, then $n = s^2$, $k_1+k_2 = 2(s-1)$, so $k_1\leq (s-1)$. At the same time, since $X_1$ has diameter $2$, degree $k_1$ should satisfy $k_1^{2}\geq n-1$, which gives us a contradiction. Therefore, $X_{1,2} = T(s)$. Consider 2 cases.

\noindent \textbf{Case 1.} Assume $p_{1,1}^{2}\geq k_1/30$ and $p_{2,2}^{1}\geq k_2/30$.

We can rewrite the assumptions of this case in the form $p_{1,2}^{1} = p_{1,1}^{2}\dfrac{k_2}{k_1}\geq \dfrac{k_2}{30}$ and $p_{1,2}^{2}\geq \dfrac{k_1}{30}$. We know that $X_{1,2} = T(s)$ for some $s$. Let $\mathcal{C}$ be a Delsarte clique geometry of $X_{1,2}$. Then every clique $C\in \mathcal{C}$ has size
 $$1+p_{1,1}^{i}+p_{2,2}^{i}+2p_{1,2}^{i} = 1+\lambda_i(X_{1,2})= \dfrac{k_1+k_2}{2}+1\leq \frac{21}{20}k_1+1$$ 
 for $i\in \{1, 2\}$. Every pair of distinct vertices $x, y\in C$ with $c(x,y) = i\in \{1, 2\}$ is distinguished by at least $|C| - p_{1,1}^{i} - p_{2,2}^{i} = 2p_{1,2}^{i}+1\geq k_1/15+1\geq |C|/16$ vertices in $C$.

\noindent \textbf{Case 2.} Assume $p_{i,i}^{j}<k_i/30$ for $\{i,j\} = \{1,2\}$.

\noindent Using Corollary \ref{cor3} and the inequality $k_1\leq k_2 \leq \frac{11}{10}k_1$, we get 
\[\frac{k_i+k_j}{2} = \lambda_i(X_{1,2}) = p_{i,i}^{i}+2p_{i,j}^{i}+p_{j,j}^{i}\leq p_{i,i}^{i} +2\frac{11}{10}\frac{k_i}{30}+\frac{(1+\varepsilon)}{2}k_j.\]
Thus, 
\[\frac{2}{5}k_i\leq \left(\frac{1}{2} - \frac{11}{150} - \frac{11\varepsilon}{20}\right)k_i \leq p_{i,i}^{i}.\]

Therefore, by Lemma \ref{assoc-clique-geom}, the graph $X_i$ is strongly regular with smallest eigenvalue $-2$, or $X_i$ is the line graph of a regular triangle-free graph. 

Assume that for some $i\in \{1,2\}$ the graph $X_i$ is strongly regular with smallest eigenvalue~$-2$, then $X_i$, as well as $X_{1,2}$, is either $T(s)$ or $L_{2}(s)$. Since $X_i$ and $X_{1,2}$ have the same number of vertices, the only possibility is $X_{1,2} = T(s_1)$ and $X_{i} = L_2(s_2)$. Then $s_1(s_1 - 1)/{2} = s_2^{2}$, so $\sqrt{2}s_2> (s_1-1)$. This implies 
\[k_i+k_j = 2(s_1 - 2) \leq 2\sqrt{2}(s_2 - 1)+1  = \sqrt{2}k_i+1,\]
 so $k_j \leq (\sqrt{2} - 1)k_i+1$ and we get a contradiction with $k_i\leq \frac{11}{10}k_j$ and $k_i\geq \sqrt{n-1}>3$.

\begin{remark}\label{remark-x1-x12}
Observe that the argument in the last paragraph of the proof shows that $X_1$ and $X_{1,2}$ cannot be simultaneously strongly regular with smallest eigenvalue $-2$ even if the assumption that $k_2\leq \frac{11}{10}k_1$ does not hold (we assume that all other assumptions of the theorem are satisfied).
\end{remark}
 
\end{proof}

\subsection{Constituent $X_2$ is strongly regular}
Next, we consider the case when $X_2$ is strongly regular, i.e., we assume that the assumptions of the statement 4 of Theorem~\ref{assoc-diam2} hold.

\begin{theorem}\label{assoc-x2-strongly-reg}
Let $\mathfrak{X}$ be an association scheme  of rank 4 and diameter~2 on $n\geq 29$ vertices. Assume additionally, that the constituents of $\mathfrak{X}$ are ordered by degree and the inequality $ k_2\leq {\varepsilon} k_3/2$ holds for some $\varepsilon < 10^{-11}$. Suppose that $k_2\leq \frac{101}{100}k_1$ and $X_2$ is strongly regular with smallest eigenvalue $-2$. Then 
 \[q(X_{1,2})+\xi(X_{1,2})\leq \frac{99}{100}(k_1+k_2).\]
\end{theorem}
\begin{proof} The assumptions of the theorem imply that the inequalities from Lemma~\ref{assoc-param-ineq} hold.

Since $X_{2}$ is strongly regular with smallest eigenvalue $-2$ and $n\geq 29$, by Seidel's classification (Theorem \ref{geom-eig-2}), by the discussion in Section~\ref{sec-subsec-exceptions}, and by Lemma~\ref{assoc-param-ineq}, we conclude that 
\begin{equation}\label{eq-x2-seidel-conseq}
{k_2}/{2}\geq p_{2,2}^{2} \geq {k_2}/{2} - 1\quad \text{and} \quad p_{2,2}^{1} = p_{2,2}^{3} \leq \varepsilon k_2.
\end{equation}
By Proposition \ref{assoc-x12-approx}, for $\varepsilon_1 = 25\varepsilon^{1/3}$, we have 
\begin{equation}\label{eq-x2-xi}
\xi(X_{1,2}) \leq \frac{p_{2,2}^{2}+p_{2,1}^{1}+p_{1,1}^{1}+\sqrt{(p_{2,1}^{1}+p_{1,1}^{1} - p_{2,2}^{2})^2+4p_{1,1}^{2}p_{2,1}^{1}}}{2}+\varepsilon_{1}(k_1+k_2), \, \text{ so}
\end{equation}
\begin{equation}\label{eq-x2-xi-2}
\begin{aligned}
\xi(X_{1,2}) & \leq \max(p_{2,2}^{2}, p_{2,1}^{1}+p_{1,1}^{1})+\sqrt{p_{1,1}^{2}p_{1,2}^{1}}+\varepsilon_{1}(k_1+k_2)\leq \\
 & \leq \max(p_{2,2}^{2}+p_{2,1}^{1}, 2p_{2,1}^{1}+p_{1,1}^{1})+\varepsilon_{1}(k_1+k_2)\leq \\
& \leq \max(p_{2,2}^{2}+p_{2,1}^{1}, \lambda_1(X_{1,2}))+\varepsilon_{1}(k_1+k_2).
\end{aligned} 
\end{equation}
Recall also, that similarly as in Eq. \eqref{eq-mux12-above} and Eq. \eqref{eq-lambda-above}, we have
\begin{equation}\label{eq-x12-lambda-mu-resrt}
\mu(X_{1,2})\leq 2\varepsilon(k_1+k_2)\quad  \text{and} \quad \max(\lambda_1(X_{1, 2}), \lambda_2(X_{1,2}))\leq \frac{1+2\varepsilon}{2}(k_1+k_2).
\end{equation}

\noindent \textbf{Case 1.} Assume $p_{1,2}^{1}>2 k_1 /5$.

\noindent Then, using that $k_2\leq \frac{101}{100}k_1\leq \frac{11}{10}k_1$, and using Eq.~\eqref{eq-x2-seidel-conseq}, 
\[\lambda_1(X_{1,2}) = p_{1,1}^{1}+2p_{1,2}^{1}+p_{2,2}^{1}\geq \frac{4}{5}k_1\geq \frac{4}{11}(k_1+k_2),\text{ and}\]  
\[\lambda_2(X_{1,2}) = p_{2,2}^{2}+2p_{1,2}^{2}+p_{1,1}^{2}\geq \frac{k_2}{2}+\frac{2}{5}\frac{10}{11} k_1\geq \frac{4}{11}(k_1+k_2).\] 
Since Eq. \eqref{eq-x12-lambda-mu-resrt} holds, $X_{1,2}$ satisfies the assumptions of Theorem \ref{Metsch} for $m=2$. So, by Lemma \ref{assoc-clique-geom}, it is strongly regular with smallest eigenvalue $-2$. However, by Theorem~\ref{assoc-x3-strongly-reg}, under the assumptions of this proposition $X_{1,2}$ and $X_2$ cannot be strongly regular with smallest eigenvalue $-2$ simultaneously. Hence, this case is impossible.

\noindent \textbf{Case 2.} Assume $2 k_1/5\geq p_{1,2}^{1}\geq k_1/8$.

\noindent \textbf{Case 2.a} Suppose $\lambda_2(X_{1,2}) \geq \lambda_1(X_{1,2})$.

\noindent Then, since $k_2\leq \frac{11}{10} k_1$, using Eq.~\eqref{eq-x2-seidel-conseq}, the inequality 
\[q(X_{1,2}) \leq \lambda_2(X_{1,2}) = p_{2,2}^{2}+2p_{1,2}^{2}+p_{1,1}^{2} \leq \frac{k_2}{2}+2\varepsilon k_1 + \frac{2}{5}k_1 \leq \frac{49}{100}(k_1+k_2)\]
 holds. At the same time, by Eq. \eqref{eq-x2-xi-2}, we get 
\[\xi(X_{1,2}) \leq \max\left(\frac{1}{2}k_2+\frac{2}{5}k_1, \lambda_1(X_{1,2})\right)+\varepsilon_1(k_1+k_2)\leq \]
\[\leq \max\left(\frac{49}{100}(k_1+k_2), \lambda_2(X_{1,2})\right)+\varepsilon_1(k_1+k_2) \leq \frac{1}{2}(k_1+k_2).\] 
Therefore, $q(X_{1,2})+\xi(X_{1,2})\leq \frac{99}{100}(k_1+k_2)$.

\noindent \textbf{Case 2.b} Suppose $\lambda_1(X_{1,2}) \geq \lambda_2(X_{1,2})$.

\noindent We may assume that $\lambda_1(X_{1,2})\geq \frac{49}{100}(k_1+k_2)$. Otherwise, by Eq. \eqref{eq-x2-xi-2},
 \[q(X_{1,2})+\xi(X_{1,2})\leq \lambda_1(X_{1,2})+\max\left(\frac{k_2}{2}+\frac{2 k_1}{5}, \lambda_1(X_{1,2})\right)+\varepsilon_1(k_1+k_2)\leq \frac{99}{100}(k_1+k_2).\] 
 Let $p_{1,2}^{1} = \alpha k_1$, then  $1/8\leq \alpha\leq 2/5$. The inequality 
 \[p_{1,1}^{1}+p_{2,2}^{1}+2p_{1,2}^{1} = \lambda_1(X_{1,2})\geq \frac{49}{100}(k_1+k_2)\]
and Eq.~\eqref{eq-x2-seidel-conseq}  imply that 
 \begin{equation}\label{eq-x2-p111-bel}
  p_{1,1}^{1}+ p_{1,2}^{1} \geq \frac{49}{100}(k_1+k_2)-\varepsilon k_2 - \alpha k_1\geq \left(\frac{49}{50} - 2\varepsilon-\alpha\right) k_1\geq \frac{28}{50}k_1\geq \frac{28}{55}k_2>p_{2,2}^{2}.
 \end{equation}
On the other hand, Eq. \eqref{eq-x12-lambda-mu-resrt} implies
\begin{equation}\label{eq-x2-p111-ab}
 p_{1,1}^{1}+p_{1,2}^{1} \leq \lambda_1(X_{1,2}) - p_{1,2}^{1}\leq  \left(\frac{1}{2}+\varepsilon\right)(k_1+k_2) - \alpha k_1\leq \left(\frac{101}{100} - \alpha \right)k_1 . 
\end{equation}
Hence, as $p_{2,2}^{2}\geq (1/2-\varepsilon)k_2\geq (1/2-\varepsilon)k_1$,  Eq. \eqref{eq-x2-p111-bel} and \eqref{eq-x2-p111-ab} imply that
\[\left\vert p_{1,1}^{1} +p_{1,2}^{1}-p_{2,2}^{2}\right\vert\leq \left(\frac{101}{100} - \alpha \right)k_1 - \left(\frac{1}{2}-\varepsilon\right)k_1 \leq \left(\frac{52}{100} - \alpha\right)k_1.\]
Therefore, using Eq.~\eqref{eq-x2-xi}, Eq.~\eqref{eq-x2-p111-ab} and $p_{2,2}^{2}\leq k_2/2\leq \frac{101}{200}k_1$, we get for $1/8\leq \alpha \leq 2/5$
\begin{equation}
\xi(X_{1,2}) \leq \frac{\frac{101}{200}+\left(\frac{101}{100} - \alpha\right)+\sqrt{\left(\frac{52}{100} - \alpha\right)^{2} +4\alpha^{2}}}{2}k_1+\frac{201}{100}\varepsilon_1 k_1\leq \frac{195}{200}k_1, 
\end{equation}
Thus, $$q(X_{1,2})+\xi(X_{1,2})\leq \frac{1+2\varepsilon}{2}(k_1+k_2)+\frac{195}{200}k_1\leq \frac{99}{100}(k_1+k_2).$$

\noindent \textbf{Case 3.} Assume $p_{1,2}^{1}< k_1/8$.

\noindent Then, using Eq.~\eqref{eq-x2-xi-2}, Corollary~\ref{cor3} and inequality $k_2\leq \frac{101}{100}k_1$,  
\[\xi(X_{1,2})\leq \max(p_{2,2}^{2}+p_{2,1}^{1}, 2p_{2,1}^{1}+p_{1,1}^{1})+\varepsilon_{1}(k_1+k_2)\leq \]
\[ \leq \max\left(\frac{1}{2}k_2+\frac{1}{8}k_1, \frac{1}{4}k_1+\frac{1+\varepsilon}{2}k_1 \right)+\varepsilon_1(k_1+k_2) \leq \frac{2}{5}(k_1+k_2).\]
Combining this with Eq.~\eqref{eq-x12-lambda-mu-resrt} we get
$q(X_{1,2})+\xi(X_{1,2})\leq \frac{99}{100}(k_1+k_2)$.

\end{proof}

\subsection{Constituent $X_1$ is strongly regular}\label{sec-x1}

The common strategy of our proofs is to prove a good spectral gap for a certain union of the constituents, or to apply Metsch's criteria (Theorem \ref{Metsch}) to a certain union of the constituents. The next lemma covers the range of parameters for which spectral gap is hard to achieve, and the conditions of Metsch's criteria are not satisfied for $X_{2}$ and $X_{1,2}$. However, in this range of parameters, we are still able to use the idea of Metsch's proof to show that $k_2$ does not differ much from $k_1$. This will suffice for our purposes.  
\begin{definition}
For a homogeneous configuration $\mathfrak{X}$ and disjoint non-empty sets of edge colors $I$ and $J$ we say that vertices $x, y_1, y_2, ..., y_t$ form a \textit{$t$-claw} (claw of size $t$) in colors $(I, J)$ if $c(x,y_i)\in I$ and $c(y_i, y_j)\in J$ for all distinct $1\leq i,  j\leq t$.  
\end{definition}
\begin{lemma}\label{assoc-claw-proof-bound}
Let $\mathfrak{X}$ be an association scheme of rank 4 and diameter~2 with constituents ordered by degree. Suppose that the inequality $k_2\leq {\varepsilon} k_3/2$ holds for some $0< \varepsilon \leq \frac{1}{100}$.  Assume additionally, that for some $0<\delta\leq \frac{1}{100}$ we have
\[p_{2,2}^{2}\geq \frac{1 - \delta}{2}k_2\quad  \text{and} \quad \frac{1}{8}k_2 \leq p_{2,2}^{1} \leq \frac{1}{3} k_2.\]
Then $k_2\leq 20 k_1$.
\end{lemma}
\begin{proof}
The assumptions of the lemma ensure that the inequalities from Lemma~\ref{assoc-param-ineq} hold.

First, we show that under the assumptions of the lemma there are no 3-claw in colors $(2,3)$ in $\mathfrak{X}$. That is, for $x, y_1, y_2, y_3\in \mathfrak{X}$ it is not possible that $c(x,y_i) = 2$ and $c(y_i,y_j) = 3$ for all distinct $i,j\in [3]$. Indeed, suppose such $x, y_i$ exist. Let $U_i = N_2(x)\cap N_2(y_i)$. Then
\[
|U_{i}| = p_{2,2}^{2}\geq \frac{1 - \delta}{2}k_2,\qquad  |U_{i}\cap U_{j}| \leq |N_2(y_i)\cap N_2(y_j)|= p_{2,2}^{3}\leq \varepsilon k_2 \quad \text{ and } \]
\[ 
|U_1 \cup U_2 \cup U_3|\leq |N_2(x)| = k_2.
\]
Therefore, we should have $k_2 \geq 3\dfrac{(1 - \delta)}{2}k_2-3\varepsilon k_2$, a contradiction.
Hence, the size of a maximal claw in colors $(2,3)$ is 2. 

Now, we claim that any edge of color 2 lies inside a clique of size at least $p_{2,2}^{2} - p_{2,2}^{3} - p_{2,1}^{3}$ in $X_{1,2}$. Consider any edge $\{x,y\}$ of color 2. Let $z$ be a vertex which satisfies $c(x,z) = 2$ and $c(y,z) = 3$. Define 
\begin{equation}
C(x,y) = \{x, y\}\cup\{ w: c(z,w) = 3\, \text{ and } \, c(x,w) = 2,\, c(y,w) = 2\}.
\end{equation} 
Observe that 
\begin{equation}\label{eq:C-size}
|C(x,y)|\geq 2+ p_{2,2}^{2} - p_{2,2}^{3} - p_{2,1}^{3}.
\end{equation}

 At the same time, if $z_1, z_2 \in C(x,y)$ satisfy $c(z_1, z_2) = 3$, then $x, z, z_1, z_2$ form a 3-claw in colors $(2,3)$, which contradicts our claim above. Hence, $C(x,y)$ is a clique in $X_{1,2}$.

Assume that there is an edge $\{y_1, y_2\}$ in $C(x,y)$ of color $1$ for some $x,y$. Then \[2k_1+\frac{1}{3}k_2\geq 2\sum\limits_{i = 1}^{3} p_{1,i}^{1} +p_{2,2}^{1}\geq p_{1,1}^{1}+2p_{1,2}^{1}+p_{2,2}^{1}\geq |C(x,y)|-2\geq \frac{1 - \delta - 2\varepsilon}{2}k_2 - k_1.\]
Therefore, $k_2\leq 20 k_1$.

Assume now that all edges in $C(x,y)$ are of color $2$ for all $x,y$, that is, $C(x,y)$ is a clique in $X_{2}$. Let $\mathcal{C}$ be the set of all maximal cliques in $X_2$ of size at least $p_{2,2}^{2} - p_{2,2}^{3} - p_{2,1}^{3}$. Then we have proved that every edge of color $2$ is covered by at least one clique in $\mathcal{C}$. Consider, two distinct cliques $C_1, C_2 \in \mathcal{C}$. There is a pair of vertices $v\in C_1\setminus C_2$ and $u\in C_2\setminus C_1$ with $c(v, u) \neq 2$. Thus,
\begin{equation}
|C_1\cap C_2|\leq \max(p_{2,2}^{1}, p_{2,2}^{3})\leq k_2/3.
\end{equation}

Suppose first that some pair of distinct cliques $C_1, C_2\in \mathcal{C}$ satisfies $|C_1\cap C_2|\geq 2$ and let $\{x,y\}\subseteq C_1\cap C_2$. Then $c(x,y) = 2$ and every vertex in $C_1\cup C_2$ is adjacent to both $x$ and $y$ by an edge of color $2$. Thus, 
\[p_{2,2}^{2}\geq |C_1\cup C_2| - 2 = |C_1|+|C_2| - |C_1\cap C_2| - 2 \geq  2(p_{2,2}^{2} - p_{2,2}^{3} - k_1) - \frac{1}{3}k_2,\]  
 so, using Lemma~\ref{assoc-param-ineq},
\[\left(\frac{1}{3}+2\varepsilon\right)k_2+2k_1 \geq p_{2,2}^{2}\geq \frac{1 - \delta}{2}k_2. \]
Hence, $k_2\leq 20k_1$.

Finally, if for every pair of distinct cliques $C_1, C_2\in \mathcal{C}$ we have $|C_1\cap C_2|\leq 1$, then every edge of color $2$ lies in at most one clique of $\mathcal{C}$. Above we proved that every edge of color $2$ lies in at least one clique of $\mathcal{C}$, so it lies in exactly one. 

Therefore, since $p_{2,2}^{2}\geq \dfrac{1 - \delta}{2}k_2$, we get that either $p_{2,1}^3\geq k_2/10$, and so $k_2\leq 10k_1$; or, by Eq.~\eqref{eq:C-size}, $|C|> k_2/3+1$ for every $C\in \mathcal{C}$, and so every vertex lies in at most $2$ cliques from~$\mathcal{C}$. In the latter case, by Lemma~\ref{clique-mu-bound}, we get that $p_{2,2}^{1}\leq 4$, which contradicts $p_{2,2}^{1}\geq k_2/8$, since $k_2\geq {p_{2,2}^{3}}/{\varepsilon}\geq {1}/{\varepsilon}$ by Lemma~\ref{assoc-param-ineq}.
\end{proof}

Furthermore, we can get a linear inequality between $k_1$ and $k_2$ if we know that $X_{1,2}$ has a clique geometry.
\begin{lemma}\label{assoc-k1-k2-bound}
Let $\mathfrak{X}$ be an association scheme of rank 4 on $n\geq 29$ vertices, with diameter~2 and constituents ordered by degree. Assume the inequality $k_2\leq {\varepsilon}k_3/2$ holds for some $\varepsilon<\frac{1}{10}$. Suppose $X_{1,2}$ has a clique geometry such that every vertex belongs to at most $m$ cliques. Then
\[ p_{2,3}^{1} \leq \frac{m^2 -2}{2}k_1\quad \text{ and } \quad k_2 \leq \frac{3}{2-4\varepsilon}(m^{2} - 2) k_1. \]
If, additionally,  $X_1$ is a strongly regular graph with smallest eigenvalue $-2$, then
\[ p_{2,3}^{1} \leq \frac{m^2 -4}{8}k_1\quad \text{ and } \quad k_2 \leq \frac{3}{8(1-2\varepsilon)}(m^{2} - 4) k_1. \]
\end{lemma}
\begin{proof}
By Lemma \ref{clique-mu-bound} applied to $X_{1,2}$, we know
\begin{equation}\label{eq:mu-sum}
p_{1,1}^{3}+2p_{1,2}^{3}+p_{2,2}^{3} = \mu(X_{1,2}) \leq m^2.
\end{equation}
Since $\mathfrak{X}$ is of diameter $2$, we have that $p_{1,1}^{3}\geq 1$ and $p_{2,2}^{3}\geq 1$, and $k_1(k_1 -1)\geq k_3$. Thus
\begin{equation}\label{eq-k1-k2-p123}
p_{1,2}^{3}\leq \frac{m^{2} - 2}{2}, \, \text{ so }\, p_{2,3}^{1}\leq \frac{m^2 - 2}{2}\frac{k_3}{k_1} \leq \frac{m^2 - 2}{2}k_1.
\end{equation} 
By Eq.~\eqref{eq-sum-param}, $p_{2,1}^{1}+p_{2,2}^{1}+p_{2,3}^{1} = k_2$, and Corollary \ref{cor2} implies that
$ p_{2,3}^{1}+\varepsilon k_2 \geq \max(p_{2,2}^{1}, p_{2,1}^{1})$. Thus, combining with Eq.~\eqref{eq-k1-k2-p123}, we get
 $$ \frac{1-2\varepsilon}{3}k_2 \leq p_{2,3}^{1} \leq \frac{m^2 - 2}{2}k_1.
 $$
If $X_1$ is strongly regular with smallest eigenvalue $-2$, we can get better estimates. By Seidel's classification, $X_1$ is either $T(s)$ or $L_{2}(s)$ for some $s$. Thus, either $n = {s(s-1)}/{2}$ and $k_1 = 2(s-2)$, or $n = s^2$ and $k_1 = 2(s-1)$. In any case, $4k_3\leq k_1^{2}$.  Observe that Corollary~\ref{cor2} implies 
$ p_{2,3}^{2}+\varepsilon k_2 \geq \max(p_{2,2}^{2}, p_{2,1}^{2}) $. Hence, $p_{2,3}^{2} \geq \dfrac{1-2\varepsilon}{3}k_2$, so $$\displaystyle{p_{2,2}^{3}\geq \frac{(1-2\varepsilon)(k_2)^2}{3k_3}\geq \dfrac{4(1-2\varepsilon)}{3}}.$$
 At the same time, $p^3_{1,1}=\mu(X_1)\geq 2$ for $X_1 = T(s)$, or $X_1 = L_2(s)$. Thus, $p_{i,i}^{3}\geq 2$ for $i=1$ and $i=2$. Therefore, as in Eq. \eqref{eq-k1-k2-p123}, by Eq.~\eqref{eq:mu-sum},
\[ p_{2,3}^{1}\leq \frac{m^2 - 4}{2}\cdot \frac{k_3}{k_1} \leq \frac{m^2 - 4}{8}k_1. \]
Again, $p_{2,3}^{1}\geq  \dfrac{1-2\varepsilon}{3}k_2$ implies the desired inequality between $k_1$ and $k_2$. 
\end{proof} 

Now, we are ready to consider the case when the constituent $X_1$ is strongly regular (case of statement 3 of Theorem~\ref{assoc-diam2}). 

\begin{theorem}\label{assoc-x1-strongly-reg}
Let $\mathfrak{X}$ be an association scheme of rank 4 on $n\geq 29$ vertices  with diameter~2 and constituents ordered by degree. Assume additionally, that the parameters of $\mathfrak{X}$ satisfy $k_2\leq {\varepsilon} k_3/2$ for $\varepsilon = 10^{-26}$. Suppose that $X_1$ is a strongly regular graph with smallest eigenvalue $-2$. Then
\begin{equation}\label{eq-x1-goal}
 q(Y)+\xi(Y)\leq (1-\varepsilon)k_Y,
 \end{equation}
where either $Y = X_{2}$ and $k_Y = k_2$, or $Y = X_{1,2}$ and $k_Y = k_1+k_2$.

\end{theorem}
\begin{proof}
The assumptions of the lemma ensure that the inequalities from Lemma~\ref{assoc-param-ineq} hold.

Since $X_1$ is strongly regular with smallest eigenvalue $-2$, Seidel's classification and Lemma~\ref{assoc-param-ineq} implies that
\begin{equation}\label{eq-x1-seidel-conseq}
  \dfrac{1}{2}k_1 \geq p_{1,1}^1\geq \dfrac{1}{2}k_1-1 \geq \left(\frac{1}{2}-\varepsilon\right)k_1 \quad \text{and} \quad p_{1,1}^2 =p_{1,1}^3\leq \varepsilon k_1 
  \end{equation}
By Lemma \ref{assoc-x12-approx}, for $\varepsilon_1 = 25\varepsilon^{1/3}\leq \frac{2}{3}10^{-7}$ we have 
\begin{equation}\label{eq-x12-spectral-gap-in-x1}
\xi(X_{1,2}) \leq \frac{p_{1,1}^{1}+p_{1,2}^{2}+p_{2,2}^{2}+\sqrt{(p_{1,2}^{2}+p_{2,2}^{2}-p_{1,1}^{1})^{2}+4p_{2,2}^{1}p_{1,2}^{2}}}{2}+\varepsilon_1(k_1+k_2).
\end{equation} 
Since $\sqrt{x+y}\leq \sqrt{x}+\sqrt{y}$,
\begin{equation}
\xi(X_{1,2})\leq \max(p_{1,1}^{1}, p_{2,2}^{2}+p_{2,1}^{2})+\sqrt{p_{2,2}^{1}p_{2,1}^{2}}+\varepsilon_1 (k_1+k_2).
\end{equation}
Using that $\lambda_{2}(X_{1,2}) \geq p_{2,2}^{2}+p_{2,1}^{2}$ (see Eq.~\eqref{eq:lambda-i}) and $p_{2,1}^{2} = p_{2,2}^{1}{k_1}/{k_2}\leq p_{2,2}^{1}$, we can simplify it even more
\begin{equation}
\xi(X_{1,2})\leq \max\left(\frac{k_1}{2}, \lambda_2(X_{1,2})\right)+p_{2,2}^{1}+\varepsilon_1 (k_1+k_2).
\end{equation}
Recall that as in Eq.~\eqref{eq-mux12-above} and Eq.~\eqref{eq-lambda-above}, we have
\begin{equation}\label{eq-x12-mu-lambda-resrt-in-x1}
\mu(X_{1,2})\leq 2\varepsilon(k_1+k_2)\quad  \text{and} \quad \max(\lambda_1(X_{1, 2}), \lambda_2(X_{1,2}))\leq \frac{1+2\varepsilon}{2}(k_1+k_2).
\end{equation}

\noindent \textbf{Case A.} Suppose $p_{2,2}^{2}\geq (2-2\delta) p_{2,2}^{1}$ for $\delta = 10^{-7}$.

\noindent Using Corollary \ref{cor3} for $p_{2,2}^{2}$, we get 
\begin{equation}\label{eq-p212-u}
p_{2,2}^{1}\leq \frac{1+\varepsilon}{4(1-\delta)}k_2 \quad \text{and} \quad  p_{2,1}^{2} =\frac{k_1}{k_2} p_{2,2}^{1} \leq \frac{1+\varepsilon}{4(1-\delta)}k_1.
\end{equation}
Note, by Corollary \ref{cor2} and Eq.~\eqref{eq-x1-seidel-conseq}, $\varepsilon k_1+p_{1,3}^1\geq p_{1,1}^1\geq \left(1/2-\varepsilon\right)k_1$. So, Eq.~\eqref{eq-sum-param} implies $p_{1,2}^{1}\leq 3\varepsilon k_1$. Therefore, by Eq.~\eqref{eq-x1-seidel-conseq},
\begin{equation}\label{eq-x1-lambda1-boundd}
\lambda_{1}(X_{1,2}) = p_{1,1}^{1}+p_{2,2}^{1}+2p_{1,2}^{1} \leq \frac{1}{2}k_1+p_{2,2}^1+6\varepsilon k_1.
\end{equation}
Assume that Eq. \eqref{eq-x1-goal} is not satisfied, then
\begin{equation}\label{eq-x1-caseA}
\begin{multlined}
 (1-\varepsilon)(k_1+k_2)\leq q(X_{1,2})+\xi(X_{1,2})\leq \\
 \leq \max\left(\lambda_2(X_{1,2}), \frac{1}{2}k_1+p_{2,2}^{1}+6\varepsilon k_1\right)+\max\left(\lambda_2(X_{1,2}), \frac{k_1}{2}\right)+p_{2,2}^{1}+\varepsilon_1(k_1+k_2)
 \end{multlined}
 \end{equation}
Observe, that if $\lambda_2(X_{1,2})\leq k_1/2+p_{2,2}^{1}+6\varepsilon k_1$, then using Eq. \eqref{eq-p212-u} we get a contradiction
$$(1-\varepsilon)(k_1+k_2)\leq k_1+3\frac{1+\varepsilon}{4(1-\delta)}k_2+\varepsilon_1(k_1+k_2)+6\varepsilon k_1.$$ 
Otherwise, Eq. \eqref{eq-x1-caseA} implies
 $$(1-\varepsilon-\varepsilon_1)(k_1+k_2)\leq 2\lambda_2(X_{1,2})+\frac{1+\varepsilon}{4(1-\delta)}k_2, \text{ so }\quad \lambda_2(X_{1,2}) \geq \frac{5}{6}k_1.$$
We estimate the expression under the root sign in Eq.~\eqref{eq-x12-spectral-gap-in-x1}, using that $\lambda_2(X_{1,2})\geq \dfrac{5}{6}k_1$, using Eq.~\eqref{eq-p212-u}, Eq.~\eqref{eq-x1-seidel-conseq}, and inequality $\dfrac{1+\varepsilon}{4(1-\delta)^2}\leq \dfrac{1+\varepsilon}{4}+\delta$ for $0\leq\varepsilon,  \delta \leq \dfrac{1}{2}$.

\[(p_{1,2}^{2}+p_{2,2}^{2}-p_{1,1}^{1})^{2}+4p_{2,2}^{1}p_{1,2}^{2} = (p_{1,2}^{2}+p_{2,2}^{2})^{2} - 2p_{1,1}^{1}(p_{1,2}^{2}+p_{2,2}^{2})+(p_{1,1}^{1})^{2}+4p_{2,2}^{1}p_{1,2}^{2} \leq \]
\[\leq (p_{1,2}^{2}+p_{2,2}^{2})^{2} -2p_{1,1}^{1}p_{1,2}^{2} - (1 - 2\varepsilon)k_1 p_{2,2}^{2} +\frac{(k_1)^{2}}{4} +\frac{1+\varepsilon}{2(1-\delta)^{2}}k_1 p_{2,2}^{2} \leq \]
\[ \leq  (p_{1,2}^{2}+p_{2,2}^{2})^{2} -k_1p_{1,2}^{2} +\left(\frac{(k_1)^2}{4} - \frac{1}{2}k_1p_{2,2}^{2}\right)+(3\varepsilon+2\delta)k_1k_2 \leq \]
\[ \leq  (p_{1,2}^{2}+p_{2,2}^{2})^{2} -\frac{k_1}{2}\left(\lambda_2(X_{1,2}) - \varepsilon k_1 -\frac{k_1}{2}\right)+(3\varepsilon+2\delta)k_1k_2 \leq\]
\[\leq (p_{1,2}^{2}+p_{2,2}^{2})^{2} - \frac{1}{6}(k_1)^{2} + (4\varepsilon+2\delta)k_1k_2.\]
Thus, since $\sqrt{x}$ is concave, using Eq.~\eqref{eq-x12-mu-lambda-resrt-in-x1}, and inequalities $\sqrt{y^{2} - x^{2}} \leq y-{x^{2}}/{(2y)}$ and $p_{2, 1}^{2}+p_{2, 2}^{2}\leq \lambda_{2}(X_{1,2})$, we obtain
\begin{equation}\label{eq-x12-root-est}
\sqrt{(p_{1,2}^{2}+p_{2,2}^{2}-p_{1,1}^{1})^{2}+4p_{2,2}^{1}p_{1,2}^{2}} \leq p_{1,2}^{2}+p_{2,2}^{2}- \frac{2(k_1)^{2}}{13(k_1+k_2)}+\sqrt{(4\varepsilon+2\delta)k_1k_2}.
\end{equation}
Denote $\varepsilon_4 = \sqrt{4\varepsilon+2\delta}<2^{-1}\cdot 10^{-3}$. Hence, by Eq.~\eqref{eq-x1-seidel-conseq}, Eq. \eqref{eq-x12-spectral-gap-in-x1} and Eq. \eqref{eq-x12-root-est},  
\[
\xi(X_{1,2}) \leq \frac{k_1}{4}+(p_{1,2}^{2}+p_{2,2}^{2}) - \frac{(k_1)^{2}}{13(k_1+k_2)}+\left(\frac{1}{2}\varepsilon_4\sqrt{k_1k_2}+\varepsilon_1 (k_1+k_2)\right).
\]
Using Corollary \ref{cor3} for $p_{2,2}^{2}$ and Eq. \eqref{eq-p212-u}, we get

\begin{equation}\label{eq-xi-bound-k1-k2}
\begin{multlined}
\xi(X_{1,2}) \leq \frac{k_1}{4}+\left(\frac{(1+\varepsilon)}{4(1-\delta)}k_1+\frac{(1+\varepsilon)}{2}k_2\right) - \frac{(k_1)^{2}}{13(k_1+k_2)}+\left(\frac{\varepsilon_4}{4}+\varepsilon_1\right) (k_1+k_2)\leq \\
\leq \frac{1}{2}(k_1+k_2) - \frac{(k_1)^{2}}{13(k_1+k_2)^{2}}(k_1+k_2)+ \varepsilon_5(k_1+k_2),
\end{multlined}
\end{equation}
where $\varepsilon_5 = \varepsilon_1+\frac{1}{4}\varepsilon_4+\delta+\varepsilon<6^{-1}\cdot 10^{-3}$. Thus, we want either $q(X_{1,2})$ to be bounded away from $(k_1+k_2)/{2}$, or to have $k_1\leq c k_2$ for some absolute constant $c$.

Observe, by Eq.~\eqref{eq-p212-u}, Eq.~\eqref{eq-x1-lambda1-boundd}  and Eq.~\eqref{eq-xi-bound-k1-k2},
\begin{equation}\label{eq-xi-lambda-1-k1-k2}
\begin{multlined}
\lambda_1 (X_{1,2}) +\xi(X_{1,2}) \leq \left(\frac{k_1}{2}+p_{2,2}^{1}+2\varepsilon k_1\right) + \frac{k_1+k_2}{2}+\varepsilon_5(k_1+k_2)\leq \\
\leq k_1+\frac{3}{4}k_2+(3\varepsilon+\delta)(k_1+k_2)+\varepsilon_5k_2 
\leq (1-\varepsilon)(k_1+k_2),
\end{multlined}
\end{equation}
\begin{equation}
\lambda_{2}(X_{1,2})+\xi(X_{1,2}) \leq \lambda_{2}(X_{1,2})+\frac{k_1+k_2}{2} +\varepsilon_5(k_1+k_2).
\end{equation}
Clearly, 
$$\mu(X_{1,2})+\xi(X_{1,2}) \leq 2\varepsilon(k_1+k_2)+\xi(X_{1,2})\leq (1-\varepsilon)(k_1+k_2).$$
Thus, either we have $q(X_{1,2})+\xi(X_{1,2})\leq (1-\varepsilon)(k_1+k_2)$, or $\lambda_2(X_{1,2})\geq (1/2-\varepsilon_5 - \varepsilon)(k_1+k_2)$.

\noindent Suppose that $\lambda_2(X_{1,2})\geq (1/2-\varepsilon_5 - \varepsilon)(k_1+k_2)$. By the assumption of Case A, we have $p^{2}_{2, 2}\geq (2-2\delta)p^{2}_{1,2}k_2/k_1$, so Eq~\eqref{eq-x1-seidel-conseq}  implies
$$\lambda_2(X_{1,2}) = p_{2,2}^{2}+2p^{2}_{1,2}+p^{2}_{1,1} \leq p_{2,2}^{2}+\frac{1}{1-\delta}\frac{k_1}{k_2}p_{2,2}^{2}+\varepsilon k_1\leq \left(\frac{1}{1-\delta}\frac{p_{2,2}^{2}}{k_2}+\varepsilon\right)(k_1+k_2).$$
Hence, in this case 
\begin{equation}\label{eq-caseA-p222-bound}
p_{2,2}^{2}\geq \left(\frac{1}{2} - \varepsilon_5 - 2\varepsilon\right)(1-\delta)k_2 .
\end{equation}

\begin{enumerate}
\item[]\textbf{Case A.1} Assume $p_{2,2}^{1}< k_2/8$.

\noindent Then $\mu(X_{2})\leq k_2/8$ and $\lambda(X_2) = p_{2,2}^{2}$. Therefore, $X_2$ satisfies the assumptions of Theorem~\ref{Metsch} for $m=2$. Thus, by Lemma \ref{clique-mu-bound} for graph $X_2$, we get $p_{2,2}^{3}\leq m^2 = 4$. At the same time, by Crorllary \ref{cor2} and Eq. \eqref{eq-sum-param} we have $p_{2,3}^{2}\geq \dfrac{1 - 2\varepsilon}{3}k_2$. Therefore, \[4k_3\geq p_{2, 2}^{3}k_3 = p_{2,3}^{2}k_2\geq k_2\dfrac{(1 - 2\varepsilon)}{3}k_2\geq \frac{1}{4}(k_2)^{2}.\]
Combining with $(k_1)^{2}\geq k_3$, we obtain $k_2\leq 4k_1$.
\item[]\textbf{Case A.2} Assume $p_{2,2}^{1}\geq k_2/8$.

\noindent Then, since Eq.~\eqref{eq-p212-u} and Eq.~\eqref{eq-caseA-p222-bound} hold, by Lemma~\ref{assoc-claw-proof-bound},   we get that $k_2\leq 20k_1$. 
\end{enumerate}

\noindent Hence, Eq. \eqref{eq-xi-bound-k1-k2} and Eq. \eqref{eq-x12-mu-lambda-resrt-in-x1} imply
\[\lambda_{2}(X_{1,2})+\xi(X_{1,2})\leq (1-\varepsilon)(k_1+k_2).\]
Therefore, using bound on $\mu(X_{1,2})$ and Eq. \eqref{eq-xi-lambda-1-k1-k2}, we get $q(X_{1,2})+\xi(X_{1,2})\leq (1-\varepsilon)(k_1+k_2)$.

\noindent \textbf{Case B.} Suppose $p_{2,2}^{2}\leq (2-2\delta) p_{2,2}^{1}$.

\noindent In this case, in several ranges of parameters we will show that $X_{1,2}$ has a clique geometry. We first establish the following bounds. 

By Corollary \ref{cor3}, $2p_{2,1}^{2}\leq (1+\varepsilon)k_1$. Assume $p_{2,2}^{1}\geq k_2/5$ and $m\leq 5$, then Eq.~\eqref{eq-x12-mu-lambda-resrt-in-x1} and Eq.~\eqref{eq-x1-seidel-conseq} implies
\begin{equation}\label{eq-x1-l1bound}
2\lambda_1(X_{1,2})-\lambda_{2}(X_{1,2})\geq k_1 +2p_{2,2}^{1} - p_{2,2}^{2} - 2p_{21}^{2} - 3\varepsilon k_1 \geq 
2\delta p_{2,2}^{1} - 4\varepsilon k_1\geq (2m-1)\mu(X_{1,2}).
\end{equation}
Suppose that $\lambda_2(X_{1,2})\geq (1/4+2m\varepsilon)(k_1+k_2)$, then Eq. \eqref{eq-x12-mu-lambda-resrt-in-x1} implies
\begin{equation}\label{eq-x1-l2bound}
2\lambda_2(X_{1,2}) - \lambda_1(X_{1,2})\geq (2m-1)\mu(X_{1,2}).
\end{equation}

\noindent \textbf{Case B.1.} Assume $p_{2,2}^{1}\geq {k_2}/{3}$.

\noindent Then, by Eq.~\eqref{eq-x1-seidel-conseq}, 
\begin{equation}\label{eq-x1-l1-13}
\lambda_1(X_{1,2})\geq p_{1,1}^{1}+p_{2,2}^{1} \geq \left(\frac{1}{2} - \varepsilon\right)k_1+\frac{1}{3}k_2\geq \frac{1}{3}(k_1+k_2).
\end{equation}
\begin{enumerate}
\item[]\textbf{Case B.1.a.} Suppose $p_{2,2}^{2}\geq k_2/3$.

Then $\lambda_2(X_{1,2})\geq (k_1+k_2)/3$. Thus, in notations of Theorem \ref{Metsch} we get for $X_{1,2}$ that $4\lambda^{(1)} - 6\mu(X_{1,2})\geq k_1+k_2$, and by Eq. \eqref{eq-x1-l1bound}-\eqref{eq-x1-l2bound}, inequality
$2\lambda^{(1)} - \lambda^{(2)}\geq 5\mu(X_{1,2})$ holds.
 Hence, by Theorem \ref{Metsch}, the graph $X_{1,2}$ has a clique geometry with $m=3$. Thus, by Lemma \ref{assoc-k1-k2-bound}, we have  $k_2 \leq \dfrac{15}{8(1-2\varepsilon)}k_1\leq 2k_1$. Therefore,
\[\lambda_{1}(X_{1,2})\geq p_{1,1}^{1}+p_{2,2}^{1}\geq \left(\frac{1}{2} - \varepsilon\right)k_1+\frac{1}{3}k_2 >\left(\frac{1}{3}+4\varepsilon\right)(k_1+k_2),\]
\[\lambda_{2}(X_{1,2})\geq p_{2,2}^{2}+2p_{1,2}^{2}\geq \frac{k_2}{3}+\frac{2k_1}{3}> \left(\frac{1}{3}+4\varepsilon\right)(k_1+k_2).\]

Therefore, $X_{1,2}$ satisfies Theorem \ref{Metsch} for $m = 2$, and so by Lemma \ref{assoc-clique-geom}, it is strongly regular with smallest eigenvalue $-2$. However,  by Theorem~\ref{assoc-x3-strongly-reg} and Remark~\ref{remark-x1-x12}, under the assumptions of this theorem the graphs $X_1$ and $X_{1,2}$ cannot be simultaneously strongly regular with smallest eigenvalue $-2$.

\item[]\textbf{Case B.1.b.} Suppose $p_{2,2}^{2}< k_2/3$.

Then, in particular, $p^{1}_{2,2}\geq p_{2,2}^{2}$, so $q(X_2) = p_{2,2}^{1}$. Take $0\leq \alpha  \leq \dfrac{1+\varepsilon}{2}\leq \dfrac{51}{100}$, and $0\leq \gamma \leq 1$, so that $p_{2,2}^{1} = \alpha k_2$ and $k_1 = \gamma k_2$. Using Eq. \eqref{x2-eq-bound} and Eq.~\eqref{eq-x1-seidel-conseq}, compute
\begin{equation}
q(X_2)+\xi(X_2) \leq p_{2,2}^{1}+p_{2,2}^{2}+\varepsilon k_2+ \sqrt{p_{2,2}^{1}p_{1,2}^{2}} +\varepsilon_1 k_2 = p_{2,2}^{2}+(\alpha+\alpha\sqrt{\gamma}+\varepsilon+\varepsilon_1)k_2
\end{equation}

If $p_{2,2}^{2} \leq (1-\alpha(1+\sqrt{\gamma})-\varepsilon_1 - 2\varepsilon) k_2$, then $q(X_2)+\xi(X_2)\leq (1-\varepsilon)k_2$ and we reached our goal. So, assume that $p_{2,2}^{2} \geq (1-\alpha(1+\sqrt{\gamma})-\varepsilon_1 - 2\varepsilon) k_2$. We compute
\begin{equation}
\begin{multlined}
\lambda_2(X_{1,2}) = p_{2,2}^{2}+2p_{1,2}^{2}+p_{1,1}^{2} \geq (1-\alpha(1+\sqrt{\gamma})-\varepsilon_1 - 2\varepsilon) k_2+2\alpha\gamma k_2 \geq \\
\geq \left(\frac{1 - \alpha(1+\sqrt{\gamma}-2\gamma) - \varepsilon_1 - 2\varepsilon}{1+\gamma}\right)(k_1+k_2) \geq \frac{3}{10}(k_1+k_2),
\end{multlined}
\end{equation}
where we use that $1+\sqrt{\gamma} - 2\gamma\geq 0$ for $0\leq \gamma \leq 1$, so expression is minimized for $\alpha = (1+\varepsilon)/{2}$ and after that we compute the minimum of the expression for $0\leq\gamma \leq 1$. Thus, by Eq. \eqref{eq-x1-l1bound}-\eqref{eq-x1-l1-13}, the graph $X_{1,2}$ has a clique geometry for $m = 3$. Hence, by Lemma~\ref{assoc-k1-k2-bound}, we have $k_2\leq 2k_1$. This implies that $\frac{1}{2}\leq \gamma \leq 1$. We compute,
\begin{equation}
\begin{gathered}
\min_{1/2\leq \gamma\leq 1}\min_{0\leq\alpha\leq \frac{51}{100}}\left(\frac{1 - \alpha(1+\sqrt{\gamma}-2\gamma) - \varepsilon_1 - 2\varepsilon}{1+\gamma}\right)= \\
= \min_{1/2\leq \gamma\leq 1}\left(\frac{1 - \frac{51}{100}(1+\sqrt{\gamma}-2\gamma) - \varepsilon_1 - 2\varepsilon}{1+\gamma}\right)\geq \frac{9}{25}>\frac{1}{3}+2\varepsilon.
\end{gathered}
\end{equation}

Therefore, using also Eq.~\eqref{eq-x1-l1bound}-\eqref{eq-x1-l1-13}, we get that $X_{1,2}$ satisfies conditions of Theorem~\ref{Metsch} for $m=2$, so by Lemma \ref{assoc-clique-geom}, the graph $X_{1,2}$ is strongly regular with smallest eigenvalue $-2$. However, by Theorem~\ref{assoc-x3-strongly-reg} and Remark~\ref{remark-x1-x12}, this is impossible, since $X_1$ is also strongly regular with smallest eigenvalue $-2$.
\end{enumerate}

\noindent \textbf{Case B.2.} Assume ${k_2}/{3}\geq p_{2,2}^{1}\geq  {k_2}/{5}$.

\noindent Then 
\begin{equation}\label{eq-x1-l1-b2}
\lambda_{1}(X_{1,2})\geq p_{1,1}^{1}+p_{2,2}^{1}\geq \left(\frac{1}{2}-\varepsilon\right)k_1 +\frac{1}{5}k_2.
\end{equation}
If $p_{2,2}^{2}\leq (1/3 - \varepsilon -\varepsilon_1) k_2$, then by Eq.~\eqref{x2-eq-bound},
\begin{equation}
\begin{aligned}
q(X_2)+\xi(X_2) &\leq \max(p_{2,2}^{2}, p_{2,2}^{1})+p_{2,2}^{2}+\sqrt{p_{2,2}^{1}p_{1,2}^{2}}+\varepsilon_1 k_2\leq \\
& \leq \frac{k_2}{3}+\left(\frac{1}{3} - \varepsilon -\varepsilon_1\right) k_2 + \frac{k_2}{3} +\varepsilon_1 k_2 \leq (1-\varepsilon)k_2.
\end{aligned}
\end{equation}
Else, $p_{2,2}^{2}\geq (1/3 - \varepsilon -\varepsilon_1) k_2\geq (1/4+10\varepsilon) k_2$, so
\begin{equation}\label{eq-x1-l2-b2}
\lambda_2(X_{1,2})\geq p_{2,2}^{2}+2p_{1,2}^{2}\geq \left(\frac{1}{4}+10\varepsilon\right)k_2+\frac{2}{5}k_1.
\end{equation}
Thus, Eq.~\eqref{eq-x1-l1bound}-\eqref{eq-x1-l2bound} and Eq. \eqref{eq-x1-l1-b2}-\eqref{eq-x1-l2-b2} imply, using Theorem \ref{Metsch}, that the graph $X_{1,2}$ has a clique geometry with $m = 5$. Therefore, using Eq.~\eqref{eq-x1-seidel-conseq} and Eq.~\eqref{eq-sum-param}, by Lemma \ref{assoc-k1-k2-bound}, 
$$\left(\frac{2}{3} - \varepsilon\right)k_2\leq (1 - \varepsilon)k_2 - p_{2,2}^{1}\leq p_{2,3}^{1}\leq \frac{m^2 - 4}{8}k_1,\, \text{ so }\, k_2\leq 4 k_1.$$
Hence, in fact, Eq. \eqref{eq-x1-l1-b2} implies  
\[\lambda_{1}(X_{1,2})\geq \frac{1}{5}k_2+\left(\frac{1}{2}-\varepsilon\right)k_1 \geq \left(\frac{1}{4}+6\varepsilon\right)(k_1+k_2).\]
Thus, using Eq. \eqref{eq-x1-l2-b2} and Eq. \eqref{eq-x1-l1bound} - \eqref{eq-x1-l2bound}, by Theorem \ref{Metsch}, we get that $X_{1,2}$ has a clique geometry for $m = 3$. Thus, we can get a better estimate, as

$$\left(\frac{2}{3} - \varepsilon\right)k_2\leq \frac{m^2 - 4}{8}k_1,\quad \text{ implies }\quad k_2\leq \frac{15}{16(1-2\varepsilon)} k_1<k_1.$$

However, this contradicts our assumption that $k_2\geq k_1$, so $p_{2,2}^{2}\geq (1/3 - \varepsilon -\varepsilon_1) k_2$ is impossible in this case.

\noindent \textbf{Case B.3.} Assume $p_{2,2}^{1}\leq {k_2}/{5}$.

\noindent Then, by the assumption of Case B, $p_{2,2}^{2}\leq (2-2\delta) p_{2,2}^{1}\leq (2-2\delta)k_2/5$, so 
\begin{equation}
\begin{multlined}
q(X_2)+\xi(X_2) \leq \max(p_{2,2}^{2}, p_{2,2}^{1}, p_{2,2}^{3})+p_{2,2}^{2}+\sqrt{p_{2,2}^{1}p_{1,2}^{2}}+\varepsilon_1 k_2\leq 
\\ \leq 2(2-2\delta)\frac{k_2}{5}+ \frac{k_2}{5} +\varepsilon_1 k_2  
\leq \left(1 - \frac{4}{5}\delta+\varepsilon_1\right)k_2\leq (1-\varepsilon)k_2.
\end{multlined}
\end{equation}

\end{proof}

\subsection{Constituent that is the line graph of a triangle-free regular graph}

Finally, we consider the case of the last possible outcome provided by Theorem~\ref{assoc-diam2}, the case when one of the constituents is the line graph of a regular triangle-free graph and is not strongly regular. 

First recall the following classical result due to Whitney.
\begin{theorem}[Corollary to Whitney's Theorem, \cite{Whitney}]\label{Whitney} Let $X$ be a connected graph on $n\geq 5$ vertices. Then the natural homomorphism $\phi :\Aut(X) \rightarrow \Aut(L(X))$ is an isomorphism $\Aut(L(X)) \cong \Aut(X)$. 
\end{theorem}

Observe, that the restriction on the diameter of the line graph gives quite strong bound on the degree of the base graph, as stated in the following lemma.

\begin{lemma}\label{line-so-k-big}
Let $X$ be a $k$-regular graph on $n$ vertices. If the line graph $L(X)$ has diameter $2$, then $k\geq n/8$. 
\end{lemma}
\begin{proof}
Recall that $L(X)$ has $kn/2$ vertices and degree $2(k-1)$. Since $L(X)$ has diameter $2$, the degree of the graph satisfies $4k^2 \geq 4(k-1)^2+2(k-1)+1\geq kn/2$, i.e., $k\geq n/8$.
\end{proof}

\begin{theorem}\label{assoc-line-triangle}
Let $X$ be a connected $k$-regular triangle-free graph on $n\geq 5$ vertices, where $k\geq 3$. Suppose $\mathfrak{X}$ is an association scheme of rank 4 and diameter 2 on $V(L(X)) = E(X)$, such that one of the constituents is equal to $L(X)$ and is not strongly regular. Then every pair of vertices $u,v\in X$ is distinguished by at least $n/8$ vertices. Therefore,   $\Aut(L(X))$ has order $n^{O(\log(n))}$ and the motion of $L(X)$ is at least $|V(L(X))|/16$. 
\end{theorem}
\begin{proof}
Denote the constituents of $\mathfrak{X}$ by $Y_i$, $0\leq i\leq 3$, where $Y_0$ is the diagonal constituent and $Y_1 = L(X)$.

Since $Y_1$ has diameter $2$, any induced cycle of $X$ has length at most 5. The graph $X$ is triangle-free, so every induced cycle in $X$ has length 4 or 5, and every cycle of length 4 or 5 is induced.

\noindent \textbf{Case 1.} Suppose that there are no cycle of length 5 in $X$, i.e., it is bipartite. 

Then for $v\in X$ there are no edges between vertices in $N_2(v)$. The graph $X$ is regular, and every induced cycle has length 4, so for every vertex $w\in N_2(v)$ the neighborhoods $N(w)$ and $N(v)$ coincide. Hence, as $X$ is connected, $X$ is a complete regular bipartite graph. However, in this case, $L(X)$ is strongly regular.

\noindent \textbf{Case 2.} Suppose there is a cycle of length 5.  

Let $v_1v_2v_3v_4v_5$ be any cycle of length 5. Take $u$ different from  $v_2, v_5$ and adjacent to $v_1$. Since the constituent $Y_1$ has diameter $2$, the edges $v_1u$ and $v_3v_4$ are at distance 2 in $L(X)$, thus there is one of the edges $uv_3$ or $uv_4$ in $X$. Again, $X$ is triangle free, so exactly one of them is in $X$. Without loss of generality, assume that $uv_3$ is in $X$. In particular, we get that there is a cycle of length 4 $uv_1v_2v_3$. Denote by $r_{i,j}$ the number of common neighbors of $v_i$ and $v_j$. Then, our argument shows that $r_{i, i+2}+r_{i,i+3} = k$ for every $i$, where indices are taken modulo 5. Thus, $r_{i, i+2} = k/2$ for every $1\leq i\leq 5$. 

  Observe, that $v_1v_2$ and $v_3v_4$ have exactly one common neighbor in $L(X)$. At the same time, for any cycle $u_1u_2u_3u_4$ edges $u_1u_2$ and $u_3u_4$ have exactly two common neighbors in $L(X)$. Thus, the pairs $(u_1u_2, u_3u_4)$ and $(v_1v_2, v_3v_4)$ belong to different constituents of the association scheme, say $Y_2$ and $Y_3$, respectively. Note, that the triple of edges $v_1v_2, v_2v_3, v_3v_4$ shows that $p^{1}_{1,3}$ is non-zero.

Take any $v\in X$ and $u\in N_2(v)$. Suppose that there is no $w\in N_2(v)$ adjacent to $u$. Then by regularity of $X$ we get $N(v) = N(u)$. For any $x,y \in N(v)$ the triple $vx, xu, uy$ form a triangle with side colors $(1,1,2)$ and we get a contradiction with $p^{1}_{1,3}\neq 0$. 

Hence, for every $u\in N_2(v)$ there exists $w\in N_2(v)$ adjacent to $u$. Take $x\in N(v)\cap N(u)$ and $y\in N(v)\cap N(w)$. Consider the cycle $vxuwy$, then as shown above, vertices $v$ and $u$ have exactly $k/2$ common neighbors. Thus, they are distinguished by at least $|N(u)\triangle N(v)| = 2(k-k/2) = k$ vertices.

Every pair of adjacent vertices has no common neighbors, so they are distinguished by at least $2k$ vertices. Thus, every pair of distinct vertices is distinguished by at least $k$ vertices. Therefore, by Lemma \ref{line-so-k-big}, every pair of distinct vertices is distinguished by at least $n/8$ vertices.

 By Lemma \ref{Whitney}, $\Aut(X) \cong \Aut(L(X))$ via natural inclusion $\phi$. Thus, bound on the order of $\Aut(L(X))$ follows from Lemma \ref{Babai-disting-order-group}. Let $W$ be the support of $\sigma \in \Aut(X) \cong \Aut(L(X))$. We show that every vertex in $W$ is incident to at most one edge fixed by $\sigma$. Consider an edge $e$ with ends $w_1, w_2$, where $w_1\in W$. Since $\sigma(w_1)\neq w_1$ the only possibility for $e$ to be fixed is $\sigma(w_1) = w_2$ and $\sigma(w_2) = w_1$. This, in particular implies that $w_2\in W$ as well. Every edge incident with $w_1$ and different from $e$ is sent by $\sigma$ to an edge incident with $w_2$, so is not fixed. Therefore, the support of $\phi(\sigma)\in \Aut(L(X))$ is at least \[\frac{|W|(k-1)}{2}\geq \dfrac{n}{8}\cdot \dfrac{(k -1)}{2}\geq \frac{nk}{32} = \frac{|V(L(X))|}{16}.\] 
\end{proof}

\section{Putting it all together}\label{sec-coherent-thm-subsec}

Finally, we are ready to combine the preceding results into our main theorem.

\begin{theorem}\label{main-coherent-2}
There exists an absolute constant $\gamma_4>0$ such that for every  primitive coherent configuration $\mathfrak{X}$ of rank $4$ on $n$ vertices either 
$$\motion(\mathfrak{X})\geq \gamma_4 n,$$
 or $\mathfrak{X}$ is a Cameron scheme.
\end{theorem}
\begin{proof} By taking $\gamma_4<1/100$ we may assume that $n>100$.

First, assume that there is an oriented color. Since the rank of $\mathfrak{X}$ is 4, the only possibility is to have two oriented colors $i, j= i^*$ and one undirected color $t$. It is easy to see that $X_t$ is a strongly regular graph.
For $n\geq 29$, by Babai's theorem (Theorem \ref{babai-str-reg-thm}), $\motion(X_t)\geq n/8$, or $X_t$ is  a triangular graph $T(s)$, a lattice graph $L_{2}(s)$, for some $s$, or their complement. 

The constituent $X_t$ cannot be the complement of $L_{2}(s)$, since the oriented diameter of $X_{i}$ should be $2$, which contradicts $k_i^{2}\geq n-1$. Indeed, in this case, $2k_i =k_i+k_i^{*} = 2(s-1)$, while $n = s^2$.   

Now, observe that $p_{i,i^{*}}^{i} = p_{i^{*},i}^{i} = p_{i,i}^{i}$. Moreover, by Eq. \eqref{eq-sum-param},
$$k_i+k_{i^*} = \left(p_{i,i}^{i}+p_{i,i^{*}}^{i} +p^{i}_{i, t}+p_{i, 0}^{i}\right)+ \left(p_{i^{*},i}^{i} + p_{i^*, i^*}^{i}+p^{i}_{i^*, t}\right).$$ 
Thus, using Eq. \eqref{eq-sum-param} again, $p_{i, i}^{i}+p_{i^*,i^*}^i\geq (2k_i-k_t-1)/{3}$. If $X_t$ is either $T(s)$ or $L_{2}(s)$, then $k_i = k_{i^*}> n/3$ and $k_t< n/3$ for $n>100$. Thus every pair of vertices connected by an edge of color $i$ is distinguished by at least $k_i/3\geq n/9$ vertices. Hence, by primitivity of $\mathfrak{X}$ and Lemma~\ref{babai-dist}, the motion of $\mathfrak{X}$ is at least $n/18$. In the last case, when $X_{t}$ is a complement of $T(s)$, the result follows from Lemma~\ref{sun-wilmes-tool} and the inequality $p_{i, i}^{i}+p_{i^*,i^*}^i\geq k_i/3$.

 Next, assume that all colors in $\mathfrak{X}$ are undirected, i.e., $\mathfrak{X}$ is an association scheme. Every constituent of $\mathfrak{X}$ has diameter at most $3$, as rank of $\mathfrak{X}$ is 4. Moreover, as discussed in Lemma~\ref{lem:drg-metric}, if there is a constituent of diameter $3$, then $\mathfrak{X}$ is induced by a distance-regular graph. In this case the statement follows from Theorem \ref{thm:main-motion-drg}. None of the components can have diameter $1$ as the rank is not 2. 
 
 Finally, if $\mathfrak{X}$ is an association scheme of rank 4 and diameter~2, then the statement of the theorem follows from Lemma~\ref{assoc-k2-large}, Theorems~\ref{assoc-diam2}, \ref{assoc-x3-strongly-reg}, \ref{assoc-x2-strongly-reg}, \ref{assoc-x1-strongly-reg}    and Theorem~\ref{assoc-line-triangle}, Observation~\ref{obs1} and Lemma~\ref{mixing-lemma-tool}.
\end{proof}

\section{Summary and open questions}\label{sec-summary}
 In this paper we studied the problem of classifying primitive coherent configurations with sublinear motion in the case of rank 4. Earlier, in \cite{kivva-drg, kivva-geometric} this problem was studied by the author for metric schemes of bounded rank, or equivalently for distance-regular graphs of bounded diameter. 

A significant obstacle for our approach, in the case of general primitive coherent configurations of rank $r\geq 5$, is the difficulty of spectral analysis for the constituents of the coherent configuration. Namely, for configurations of rank 4 we analyzed the spectral gap ``by hand'' through Propositions \ref{assoc-spectral-radius} and \ref{assoc-x12-approx}. For coherent configurations of higher rank we need more general techniques.

\begin{problem}
Do there exist $\varepsilon, \delta>0$ such that the following statement holds? If the minimal distinguishing number  $D_{\min}(\mathfrak{X})$ of a primitive coherent configuration $\mathfrak{X}$ satisfies $D_{\min}(\mathfrak{X})<\varepsilon n$, then the spectral gap for the symmetrization of one of the constituents $X_i$ of $\mathfrak{X}$ is $\geq \delta k_i$. What $\delta$ can be achieved? 
\end{problem}

We would like to point out, that even $\delta k_i$ spectral gap for one of the constituents is not sufficient for an efficient application of the spectral tool (Lemma \ref{mixing-lemma-tool}). However, we expect that a result of this flavor would introduce important techniques to the analysis.

We also would like to mention that there should be a reasonable hope to prove Conjecture~\ref{conj-3} for the case when no color is overwhelmingly dominant. The following result easily follows from minimal distinguishing number analysis. In particular, in the case of bounded rank it gives an $\Omega(n)$ bound on the motion.

\begin{proposition}\label{mindeg-bounded-degree}
Fix $0<\delta <1$ and an integer $r\geq 3$. Let $\mathfrak{X}$ be a primitive coherent configuration of rank $r$ on $n$ vertices. Assume that each constituent $X_i$ has degree $k_i\leq \delta n$.
Then \[\motion(\mathfrak{X})\geq D_{\min}(\mathfrak{X})\geq\frac{\min(\delta, 1-\delta)}{6(r-1)}n.\]  
\end{proposition}
\begin{proof}
The condition $k_i\leq \delta n$, for all $i$, implies that there exists a set $I$ of colors  such that $\sum\limits_{i\in I}k_i = \alpha n$ for some $\min(\delta, 1-\delta)/2 \leq \alpha \leq 1/2$. Fix any vertex $u$ of $\mathfrak{X}$. 
We want to show that for some vertex $v$ the inequality $D(u,v)\geq \alpha n/3$ holds. 

Assume this is not true. Denote $N_I(u) = \{z|\, c(u,z)\in I\}$. Let us count the number of pairs $(v,z)$ with $c(u,z)\in I$ and $c(v,z)\in I$ in two different ways. Since $\sum\limits_{i^*\in I}k_i = \alpha n$ and $z\in N_I(u)$, there are $\alpha^2 n^2$ such pairs. On the other hand, for every $v$ we have $D(u,v) \leq \alpha n/3$, so at least $2\alpha n/3$ vertices $z\in N_{I}(u)$ are paired with $v$. Therefore, the number of pairs in question is at least $n\cdot \dfrac{2\alpha}{3} n$. This contradicts the condition $0<\alpha\leq\dfrac{1}{2}$.  

Therefore, there exists a pair of vertices with $D(u,v)\geq \alpha n/3$. Finally, the configuration $\mathfrak{X}$ is primitive, so by Lemma~\ref{babai-dist} we get that $\motion(\mathfrak{X})\geq D_{\min}(\mathfrak{X})\geq \dfrac{\alpha}{3(r-1)}n$.
\end{proof}

However, when the rank is unbounded, this seemingly simple case of Conjecture \ref{conj-3} (every constituent has degree $\leq \delta n$) is still open. To avoid exceptions we relax the conjectured lower bound to $\Omega(n/\log(n))$.

\begin{conjecture}\label{conj-bound-degree}
Fix $0<\delta<1$. Let $\mathfrak{X}$ be a primitive coherent configuration on $n$ vertices. Assume that every constituent has degree $\leq \delta n$. Then $\motion(\mathfrak{X}) = \Omega(n/\log(n))$.
\end{conjecture}  

Next, we observe that Cameron schemes satisfy Conjecture \ref{conj-bound-degree}.

\begin{proposition}\label{prop-camer-nlogn}
Fix $0<\delta<1$.  Consider a Cameron group $(A_m^{(k)})^d \leq G\leq S_m\wr S_d$ acting on $n = \binom{m}{k}^d$ points and let $\mathfrak{X} = \mathfrak{X}(G)$ be the corresponding Cameron scheme. Assume that every constituent of $\mathfrak{X}$ has degree $\leq \delta n$. Then $\motion(\mathfrak{X}) = \Omega(n/\log(n))$. 
\end{proposition}
\begin{proof}
We can assume $k\leq m/2$. Note that then the rank of $\mathfrak{X}$ is equal to $kd+1$. 

Case 1. Suppose that $k\leq m/3$. Then
\[n = \binom{m}{k}^d\geq \left(\frac{m-k}{k}\right)^{kd}\geq \left(\frac{2k}{k}\right)^{kd} = 2^{kd}\]
Thus, $kd \leq \log(n)$ in this case, and the statement follows from Proposition~\ref{mindeg-bounded-degree}.

Case 2. Suppose that $m/3<k\leq m/2$. By Lemma~\ref{cameron-min-deg}, we have that as $m\rightarrow \infty$ the inequality $\motion(\mathfrak{X}) \geq \alpha n$ holds for some $\alpha>0$. At the same time, by the proof of Lemma~\ref{cameron-min-deg} we know that the motion of $\mathfrak{X}$ does not depend on $d$. Thus as $\motion(X) \geq \alpha n$ is violated just by finite number of pairs $(m,k)$, we still have $\motion(\mathfrak{X}) = \Omega(n)$ in this case.
\end{proof}

We observe that the bound in Conjecture \ref{conj-bound-degree}, if true, is nearly tight, for $\delta \in (1/e, 1)$, as the example of Hamming schemes $\mathfrak{H}(tm,m)$ with $t = -\lfloor \log(\delta) m\rfloor/m$ shows. Note that for $m\geq 3$ the Hamming scheme $\mathfrak{H}(k,m)$ is primitive.

\begin{proposition} Consider Hamming scheme $\mathfrak{H}(tm, m)$ with $\displaystyle{t= -\frac{\lfloor \log(\delta)m\rfloor}{m}}$ on $n = m^{tm}$ points, for $\delta \in (1/e, 1)$. Then its maximum constituent degree satisfies $k_{\max} \leq \delta n$ and the motion satisfies \[\motion(\mathfrak{H}(tm, m))= O\left(\frac{n\log\log(n)}{\log(n)}\right).\]
\end{proposition}
\begin{proof}
Note that since $tm<m$ the maximum degree is $k_{\max} = (m-1)^{tm}$. Then 
\[k_{\max} = \left(\frac{m-1}{m}\right)^{mt}n \leq \eee^{-t}n\leq \delta n.\]

The motion of $\mathfrak{H}(tm, m)$ is realized by a 2-cycle in the first coordinate, and is equal to $2n/m$. The number of vertices is $n = m^{mt} = \eee^{tm\log(m)}$, so $m\log(m) = \log(n)/t$. Thus $m>\log(n)/(t\log\log(n))$. Hence, 
\[\motion(\mathfrak{H}(tm, m))\leq \frac{2n\log\log(n)t}{\log(n)} = O\left(\frac{n\log\log(n)}{\log(n)}\right).\]
\end{proof}

\bibliographystyle{plain}

\bibliography{refer}

\end{document}